\documentclass[a4paper,english]{amsart}
\usepackage[latin9]{inputenc}
\usepackage{babel}
\usepackage{mathrsfs}
\usepackage{mathtools}
\usepackage{amstext}
\usepackage{amsthm}
\usepackage{amssymb}
\usepackage{esint}
\usepackage[unicode=true,
 bookmarks=true,bookmarksnumbered=true,bookmarksopen=true,bookmarksopenlevel=2,pdfborder={0 0 1},
 breaklinks=false, backref=false,colorlinks=false]
 {hyperref}
\hypersetup{
 colorlinks,linkcolor=blue,anchorcolor=blue,citecolor=blue}
\usepackage{breakurl}

\makeatletter

\newcommand{\rrvert}{\vert}
\newcommand{\rrVert}{\Vert}
\newcommand{\llvert}{\vert}
\newcommand{\llVert}{\Vert}
\numberwithin{equation}{section}
\numberwithin{figure}{section}
\theoremstyle{plain}
\newtheorem{thm}{\protect\theoremname}[section]
  \theoremstyle{plain}
  \newtheorem{prop}[thm]{\protect\propositionname}
  \theoremstyle{plain}
  \newtheorem{conj}[thm]{Conjecture}
  \theoremstyle{definition}
  \newtheorem{defn}[thm]{\protect\definitionname}
  \theoremstyle{plain}
  \newtheorem{lem}[thm]{\protect\lemmaname}
  \theoremstyle{remark}
  \newtheorem{rem}[thm]{\protect\remarkname}
  \theoremstyle{plain}
  \newtheorem{cor}[thm]{\protect\corollaryname}
  \theoremstyle{definition}
  \newtheorem{example}[thm]{Example}

\usepackage{babel}
\usepackage{a4wide}

\makeatother

  \providecommand{\corollaryname}{Corollary}
  \providecommand{\definitionname}{Definition}
  \providecommand{\lemmaname}{Lemma}
  \providecommand{\propositionname}{Proposition}
  \providecommand{\remarkname}{Remark}
\providecommand{\theoremname}{Theorem}

\begin{document}

\title[noncommutative maximal ergodic inequalities]{noncommutative maximal ergodic inequalities associated with doubling
conditions}

\date{\today}

\author{Guixiang Hong, Ben Liao and Simeng Wang}

\address{School of Mathematics and Statistics, Wuhan University, 430072 Wuhan,
	China and Hubei Key Laboratory of Computational Science, Wuhan University, 430072 Wuhan, China.}

\email{guixiang.hong@whu.edu.cn}

\address{Department of Mathematics,
	Texas A\&M University,
	College Station, Texas 77843-3368,
	USA}
\curraddr{Tencent Quantum Laboratory, Shenzhen, China.}

\email{liao@hotmail.co.uk}

\address{Universität des Saarlandes, Fachrichtung Mathematik, Postfach 151150,
66041 Saarbrücken, Germany}
\curraddr{Université Paris-Saclay, CNRS, Laboratoire de mathématiques d?Orsay, 91405, Orsay, France.}

\email{simeng.wang@universite-paris-saclay.fr}

\begin{abstract}
This paper is devoted to the study of noncommutative maximal inequalities
and ergodic theorems for group actions on von Neumann algebras. Consider
a locally compact group $G$ of polynomial growth with a symmetric compact
subset~$V$. Let $\alpha $ be a continuous action of $G$ on a von Neumann
algebra $\mathcal{M}$ by trace-preserving automorphisms. We then show that
the operators defined by
\begin{equation*}
A_{n}x=
\frac{1}{m(V^{n})}
\int _{V^{n}}\alpha _{g}x\,dm(g),\quad x\in
L_{p}( \mathcal{M}),n\in \mathbb{N},1\leq p\leq \infty ,
\end{equation*}
are of weak type $(1,1)$ and of strong type $(p,p)$ for
$1 < p<\infty $. Consequently, the sequence $(A_{n}x)_{n\geq 1}$ converges
almost uniformly for $x\in L_{p}(\mathcal{M})$ for $1\leq p<\infty $. Also,
we establish the noncommutative maximal and individual ergodic theorems
associated with more general doubling conditions, and we prove the corresponding
results for general actions on one fixed noncommutative $L_{p}$-space which
are beyond the class of Dunford--Schwartz operators considered previously
by Junge and Xu. As key ingredients, we also obtain the Hardy--Littlewood
maximal inequality on metric spaces with doubling measures in the operator-valued
setting. After the groundbreaking work of Junge and Xu on the noncommutative
Dunford--Schwartz maximal ergodic inequalities, this is the first time
that more general maximal inequalities are proved beyond Junge and Xu's
setting. Our approach is based on quantum probabilistic methods as well
as random walk theory.
\end{abstract}

\maketitle

\section{Introduction}
\label{sec1}

This article studies maximal inequalities
and ergodic theorems for group actions on noncommutative $L_{p}$-spaces.
The connection between ergodic theory and von Neumann algebras goes back
to the very beginning of the theory of operator algebras. However, the
study of individual ergodic theorems in the noncommutative setting only
took off with Lance's pioneering work \cite{lance76erg} in 1976. The topic
was then extensively investigated in a series of works of Conze, Dang-Ngoc,
K\"{u}mmerer, Yeadon, and others (see
\cite{conzedangngoc78erg,kummerer78erg,yeadon77max,jajte85nclimitbook}
and references therein). Among them, Yeadon \cite{yeadon77max} obtained
a maximal ergodic theorem in the preduals of semifinite von Neumann algebras.
But the corresponding maximal inequalities in $L_{p}$-spaces remained open
until the celebrated work of Junge and Xu \cite{jungexu07erg}, which established
the noncommutative analogue of the Dunford--Schwartz maximal ergodic theorem.
This breakthrough has motivated further research on noncommutative ergodic
theorems, such as \cite{anantharaman06freeerg,bekjan08ergpositive,hongsun16,hu08freeerg}, and
\cite{litvinov14wienerwintner}. Note that all these works essentially remain
in the class of Dunford--Schwartz operators, that is, do not go beyond
Junge and Xu's setting.

On the other hand, in classical ergodic theory, a number of significant
developments related to individual ergodic theorems for group actions have
been established in recent years. In particular, Breuillard
\cite{breuillard14polygrowth} and Tessera \cite{tessera07polygrowth} studied
the balls in groups of polynomial growth; they proved that for any invariant
metric quasi-isometric to a word metric (such as invariant Riemannian metrics
on connected nilpotent Lie groups), the balls are asymptotically invariant
and satisfy the doubling condition, and hence satisfy the individual ergodic
theorem. This settled a long-standing problem in ergodic theory since Calder\'{o}n's
classical paper \cite{calderon53erg} in 1953. Also, Lindenstrauss
\cite{lindenstrauss01ergamenable} proved the individual ergodic theorem
for tempered F{\o}lner sequences, which resolves the problem of constructing
pointwise ergodic sequences on an arbitrary amenable group. (We refer to
\cite{nevo06surveyerg} for more details.)

Thus, it is natural to extend Junge and Xu's work to actions of more general
amenable groups rather than the integer group. As in the classical case,
the first natural step would be to establish the maximal ergodic theorems
for doubling conditions. However, since we do not have an appropriate analogue
of covering lemmas in the noncommutative setting, no significant progress
has been made in this direction. In this paper we provide a new approach
to this problem. It is based on both classical and quantum probabilistic
methods, and allows us to go beyond the class of Dunford--Schwartz operators
considered by Junge and Xu.

Our main results establish the noncommutative maximal and individual ergodic
theorems for ball averages under the doubling condition. Let $G$ be a locally
compact group equipped with a right Haar measure~$m$. Recall that for an
invariant metric $d$ on $G$,\footnote{In this paper, we always assume that
	$d$ is a measurable function on $G\times G$ and $m$ is a Radon Borel measure
	with respect to $(G,d)$.} we say that $(G,d)$ satisfies the
\emph{doubling condition} if the balls
$B_{r}\coloneqq \{s\in G:d(s,e)\leq r\}$ satisfy
\begin{equation}
m(B_{2r})\leq Cm(B_{r}),\quad r>0, \label{eq:doubling def}
\end{equation}
where $C$ is a constant independent of~$r$. We say that the balls are
\emph{asymptotically invariant} under right translation if for every
$s\in G$,
\begin{equation}
\lim _{r\to \infty }
\frac{m((B_{r}s)\bigtriangleup B_{r})}{m(B_{r})}=0, \label{eq:asymp inv}
\end{equation}
where $\bigtriangleup $ denotes the usual symmetric difference of subsets.
To state the noncommutative ergodic theorems, we consider a von Neumann
algebra $\mathcal{M}$ equipped with a normal semifinite trace~$\tau$. We also consider an action $\alpha $ of $G$ on the associated
noncommutative $L_{p}$-spaces $L_{p}(\mathcal{M})$, under some mild assumptions
clarified in later sections. In particular, if $\alpha $ is a continuous
action of $G$ on $\mathcal{M}$ by $\tau $-preserving automorphisms of
$\mathcal{M}$, then $\alpha $ extends to isometric actions on the spaces
$L_{p}(\mathcal{M})$. The following is one of our main results.
\begin{thm}
	\label{thm:main intro}%
	Assume that $(G,d)$ satisfies \eqref{eq:doubling def} and
	\eqref{eq:asymp inv}. Let $\alpha $ be a continuous action of $G$ on
	$\mathcal{M}$ by $\tau $-preserving automorphisms. Let $A_{r}$ be the averaging
	operators
	\begin{equation*}
	A_{r}x=
	\frac{1}{m(B_{r})}
	\int _{B_{r}}\alpha _{s}x\,dm(s),\quad x\in
	\mathcal{M},r>0.
	\end{equation*}
	Then $(A_{r})_{r>0}$ is of weak type $(1,1)$ and of strong type
	$(p,p)$ for $1<p<\infty $. Moreover, for all $1\leq p<\infty $, the sequence
	$(A_{r}x)_{r>0}$ converges almost uniformly for
	$x\in L_{p}(\mathcal{M})$.
\end{thm}
Here we refer to Section~\ref{sect:max ineq} for the notion of weak and
strong type $(p,p)$ inequalities in the noncommutative setting. Also, the
notion of almost uniform convergence is a noncommutative analogue of the
notion of almost everywhere convergence. We refer to Definition~\ref{defn:au convergence} for the relevant definitions.

There exist a number of examples satisfying assumptions
\eqref{eq:doubling def} and \eqref{eq:asymp inv} of the above theorem,
for which we refer to the previously mentioned works
\cite{breuillard14polygrowth,nevo06surveyerg}, and
\cite{tessera07polygrowth}. In particular, if we take $G$ to be the integer
group $\mathbb{Z}$ and $d$ to be the usual word metric, then we recover
the usual ergodic average
$A_{n}=\frac{1}{2n+1}\sum _{k=-n}^{n}T^{k}$ for an invertible operator
$T$, as is treated in \cite{jungexu07erg}. More generally, we may consider
groups of polynomial growth.
\begin{thm}%
	\label{thm:main poly}
	Assume that $G$ is generated by a symmetric compact subset $V$ and is of
	polynomial growth.
	\begin{enumerate}
		\item[(1)] Fix $1<p<\infty $. Let $\alpha $ be a strongly continuous and
		uniformly bounded action of $G$ on $L_{p}(\mathcal{M})$ such that
		$\alpha _{s}$ is a positive map for each $s\in G$. Then the operators defined
		by
		\begin{equation*}
		A_{n}x=
		\frac{1}{m(V^{n})}
		\int _{V^{n}}\alpha _{s}x\,dm(s),\quad x\in
		L_{p}( \mathcal{M}),n\in \mathbb{N}
		\end{equation*}
		are of strong type $(p,p)$. The sequence $(A_{n}x)_{n\geq 1}$ converges
		bilaterally almost uniformly for $x\in L_{p}(\mathcal{M})$.
		\item[(2)] Let $\alpha $ be a strongly continuous action of $G$ on
		$\mathcal{M}$ by $\tau $-preserving automorphisms. Then the operators defined
		by
		\begin{equation*}
		A_{n}x=
		\frac{1}{m(V^{n})}
		\int _{V^{n}}\alpha _{s}x\,dm(s),\quad x\in
		\mathcal{M},n \in \mathbb{N}
		\end{equation*}
		are of weak type $(1,1)$ and of strong type $(p,p)$ for all
		$1 < p<\infty $. The sequence $(A_{n}x)_{n\geq 1}$ converges almost uniformly
		for $x\in L_{p}(\mathcal{M})$ for all $1\leq p<\infty $.
	\end{enumerate}
\end{thm}

The theorems rely on several key results obtained in this work. We address
the following subjects.

(i) \emph{Noncommutative transference principles.} Our first key ingredient
is a noncommutative variant of Calder\'{o}n's transference principle in
\cite{calderon68transference} (see also
\cite{coifmanweiss76transference,fendler98dilation}), given in
Theorems~\ref{thm:transference} and~\ref{thm:transference weak}. More precisely,
we prove that for actions by an amenable group, in order to establish the
noncommutative maximal ergodic inequalities, it suffices to show the inequalities
for translation actions on operator-valued functions. We remark that the
particular case of certain actions by $\mathbb{R}$ is also discussed in
\cite{hong17dimfree} by the first author.

(ii)
\emph{Noncommutative Hardy--Littlewood maximal inequalities on metric measure
	spaces.} As for the second key ingredient, we prove in Theorem~\ref{thm:hl max} a noncommutative extension of Hardy--Littlewood maximal
inequalities on metric measure spaces. For a doubling metric measure space
$(X,d,\mu )$, denote by $B(x,r)$ the ball with center $x$ and radius
$r$ with respect to the metric~$d$. Our result asserts that the Hardy--Littlewood
averaging operators on the $L_{p}(\mathcal{M})$-valued functions
\begin{equation*}
A_{r}f(x)=
\frac{1}{\mu (B(x,r))}
\int _{B(x,r)}f\,d\mu ,\quad f\in L_{p}
\bigl(X;L_{p}( \mathcal{M}) \bigr),x\in X,r>0
\end{equation*}
satisfy the weak type $(1,1)$ and strong type $(p,p)$ inequalities. We
remark that the classical argument via covering lemmas does not seem to
fit into this operator-valued setting. Instead, our approach is based on
the study of random dyadic systems by Naor and Tao in
\cite{naortao10randommartingale} and Hyt\"{o}nen and Kairema in
\cite{hytonenkairema12dyadic}. The key idea is to relate the desired inequality
to noncommutative martingales, and to use the available results in quantum
probability developed in \cite{cuculescu71doob} and
\cite{junge02doob}. The approach is inspired by Mei's famous work in
\cite{mei03dyadic} and \cite{mei07hardy}, which asserts that the usual
continuous BMO space is the intersection of several dyadic BMO spaces.

(iii) \emph{Domination by Markov operators.} In the study of ergodic theorems
for actions by free groups or free abelian groups, it is a key fact that
the associated ergodic averages can be dominated by the standard averaging
operators of the form $\frac{1}{n}\sum _{k=1}^{n}T^{k}$ for some map
$T$ (see \cite{brunel73ergzn,nevostein94freeerg}). Also, in
\cite{steinstromberg83maxlargen}, Stein and Str\"{o}mberg apply the Markov
semigroup with heat kernels to estimate the maximal inequalities on Euclidean
spaces with large dimensions. In this paper, we build a similar result
for groups of polynomial growth. Our approach is new and the construction
follows easily from some typical Markov chains on these groups. More precisely,
we show in Proposition~\ref{prop:walk domination} that for a group
$G$ of polynomial growth with a symmetric compact generating subset
$V\subset G$, and for an action $\alpha $ of $G$, there exists a constant
$c$ such that
\begin{equation*}
\frac{1}{m(V^{n})}
\int _{V^{n}} \alpha _{s} x \,dm(s) \leq
\frac{c}{n^{2}}
\sum _{k=1}^{2n^{2}}T^{k} x,
\quad x\geq 0,
\end{equation*}
where $T=\frac{1}{m(V)}\int _{V} \alpha _{s} \,dm(s)$. The result will help
us to improve the weak type inequalities in Theorem~\ref{thm:main poly}.

(iv) \emph{Individual ergodic theorems for $L_{p}$ representations.} In
the classical setting, the individual ergodic theorem holds for positive
contractions on $L_{p}$-spaces with one fixed $p\in (1,\infty )$ (see
\cite{ionescutulcea64ergisometry,akcoglu75erg}). The results can be also
generalized for positive power-bounded operators and more general Lamperti
operators (see, e.g., \cite{kan78erglamperti,martinreyesdelatorre88ergpowerbdd,tempelman15erglampertigroup}). However, in the noncommutative setting,
the individual ergodic theorems on $L_{p}$-spaces were only known for operators
which can be extended to $L_{1}+L_{\infty }$. In Section~\ref{sect:individual} we will develop some new methods to prove the individual
ergodic theorems for operators on one fixed $L_{p}$-space.

Apart from the above approach, we also provide in Section~\ref{sect:group} an alternative proof of Theorem~\ref{thm:main intro} for
discrete groups of polynomial growth. Compared to the previous approach,
this proof is much more group-theoretical and has its own interests. It
relies essentially on the concrete structure of groups of polynomial growth
discovered by Bass, Gromov, and Wolf.

We remark that although our results are stated in the setting of tracial
$L_{p}$-spaces, a large number of the results can be extended to the general
nontracial case without difficulty. Since the standard methods for these
generalizations are already well developed in
\cite{haagerupjx10reduction} and \cite{jungexu07erg}, we leave the details
to the reader and restrict our attention to the semifinite case for simplicity
of exposition.

We end this Introduction with a brief description of the organization of
the paper. In the next section, we recall some basics on noncommutative
maximal operators as well as actions by amenable groups. Section~\ref{sect:transference} is devoted to the proof
of the noncommutative variant of Calder\'{o}n's transference principle.
In Section~\ref{sect:max}, we prove the Hardy--Littlewood maximal inequalities
mentioned above and deduce the maximal inequalities in Theorem~\ref{thm:main intro}. We also use similar ideas to establish the ergodic
theorems for increasing sequences of compact subgroups (Theorem~\ref{thm:max subgp}). In the last part of the section, we will provide
an approach based on the random walk theory, which relates the ball averages
to the classical ergodic averages of Markov operators. In Section~\ref{sect:group}, we provide an alternative group-theoretical approach
to Theorem~\ref{thm:main poly}. In Section~\ref{sect:individual}, we discuss
the individual ergodic theorems, which prove the bilateral almost uniform
convergences in Theorem~\ref{thm:main intro}. Also, we give new results
on almost uniform convergences associated with actions on one fixed
$L_{p}$-space.

\section{Preliminaries}
\label{sec2}

\subsection{Noncommutative $L_{p}$-spaces and noncommutative maximal norms}
\label{sect:max ineq}

Throughout the paper, unless explicitly stated otherwise,
$\mathcal{M}$ will always denote a semifinite von Neumann algebra equipped
with a normal semifinite trace~$\tau$. Let $\mathcal{S}_{+}$ denote the
set of all $x\in {\mathcal{M}}_{+}$ such that
$\tau (\operatorname{supp} x)<\infty $, where $\operatorname{supp} x$ denotes the support
of~$x$. Let ${\mathcal{S}}$ be the linear span of
${\mathcal{S}}_{+}$. Given $1\leq p<\infty $, we define
\begin{equation*}
\llVert x \rrVert _{p}= \bigl[\tau \bigl( \llvert x \rrvert
^{p} \bigr) \bigr]^{1/p},\quad x\in {
	\mathcal{S}},
\end{equation*}
where $ \llvert  x \rrvert  =(x^{*}x)^{1/2}$ is the modulus of~$x$. Then
$({\mathcal{S}},  \llVert  \cdot  \rrVert  _{p})$ is a normed space, whose completion
is the noncommutative $L_{p}$-space associated with
$({\mathcal{M}},\tau )$, denoted by $L_{p}(\mathcal{M})$. As usual,
we set $L_{\infty }(\mathcal{M})={\mathcal{M}}$ equipped with the
operator norm. Let $L_{0}(\mathcal{M})$ denote the space of all
closed densely defined operators on $H$ measurable with respect to
$(\mathcal{M},\tau )$ ($H$ being the Hilbert space on which
$\mathcal{M}$ acts). Then $L_{p}(\mathcal{M})$ can
be viewed as closed densely defined operators on~$H$. We denote by
$L_{0}^{+}(\mathcal{M})$ the positive part of
$L_{0}(\mathcal{M})$, and set
$L_{p}^{+}(\mathcal{M})=L_{0}^{+}(\mathcal{M})
\cap L_{p}(\mathcal{M})$. We refer to
\cite{pisierxu2003nclp} for more information on noncommutative
$L_{p}$-spaces.

For a $\sigma $-finite measure space $(X,\Sigma ,\mu )$, we consider the
von Neumann algebraic tensor product
$L_{\infty }(X)\bar{\otimes }\mathcal{M}$ equipped with the trace
$\int \otimes \tau $, where $\int $ denotes the integral against~$\mu$. For $1\leq p<\infty $, the space
$L_{p}(L_{\infty }(X)\bar{\otimes }\mathcal{M})$ isometrically coincides
with $L_{p}(X;L_{p}(\mathcal{M}))$, the usual $L_{p}$-space of $p$-integrable
functions from $X$ to $L_{p}(\mathcal{M})$. In this paper we will not distinguish
these two notions unless specified otherwise.

Maximal norms in the noncommutative setting require a specific definition.
The subtlety is that $\sup _{n} \llvert  x_{n} \rrvert  $ does not make any sense for a sequence
$(x_{n})_{n}$ of arbitrary operators. This difficulty is overcome by considering
the spaces $L_{p}(\mathcal{M};\ell _{\infty })$, which are the
noncommutative analogues of the usual Bochner spaces
$L_{p}(X;\ell _{\infty })$. These vector-valued $L_{p}$-spaces were first
introduced by Pisier \cite{pisier1998ncvectorLp} for injective von Neumann
algebras and then extended to general von Neumann algebras by Junge
\cite{junge02doob}. The descriptions and properties below can be found
in \cite[Section~2]{jungexu07erg}. Given $1\le p\le \infty $,
$L_{p}(\mathcal{M};\ell _{\infty })$ is defined as the space of
all sequences $x=(x_{n})_{n\ge 0}$ in $L_{p}(\mathcal{M})$ which
admit a factorization of the following form: there are
$a,b\in L_{2p}(\mathcal{M})$ and a bounded sequence
$y=(y_{n})\subset L_{\infty }(\mathcal{M})$ such that
\begin{equation*}
x_{n}=ay_{n}b,\quad \forall n\ge 0.
\end{equation*}
We then define
\begin{equation*}
\llVert x \rrVert _{L_{p}(\mathcal{M};\ell _{\infty })}=\inf \bigl\{ \llVert a \rrVert _{2p}
\sup _{n\ge 0} \llVert y_{n} \rrVert _{\infty }
\llVert b \rrVert _{2p} \bigr\},
\end{equation*}
where the infimum runs over all factorizations as above. We will adopt
the convention that the norm
$ \llVert  x \rrVert  _{L_{p}(\mathcal{M};\ell _{\infty })}$ is denoted by
$   \llVert  \sup _{n}^{+}x_{n}   \rrVert  _{p}$. As an intuitive description, we remark
that a positive sequence $(x_{n})_{n\ge 0}$ of $L_{p}(\mathcal{M})$ belongs
to $L_{p}(\mathcal{M};\ell _{\infty })$ if and only if there exists
a positive $a\in L_{p}(\mathcal{M})$ such that $x_{n}\leq a$ for any
$n\geq 0$ and in this case,
\begin{equation*}
\bigl\llVert {
	\sup _{n}}^{+}x_{n} \bigr\rrVert
_{p}=\inf \bigl\{ \llVert a \rrVert _{p}: a \in
L_{p}(\mathcal{M}), a\ge 0 \text{ and } x_{n}\leq a \text{
	for any }n\geq 0 \bigr\} .
\end{equation*}
Also, we denote by $L_{p}(\mathcal{M};c_{0})$ the closure of finite sequences
in $L_{p}(\mathcal{M};\ell _{\infty })$ for
$1\le p<\infty $. On the other hand, we may also define the space
$L_{p}(\mathcal{M};\ell _{\infty }^{c})$, which is the space of
all sequences $x=(x_{n})_{n\ge 0}$ in $L_{p}(\mathcal{M})$ which
admit a factorization of the following form: there are
$a\in L_{p}(\mathcal{M})$ and
$y=(y_{n})\subset L_{\infty }(\mathcal{M})$ such that
\begin{equation*}
x_{n}=y_{n}a,\quad \forall n\geq 0.
\end{equation*}
And we define
\begin{equation*}
\llVert x \rrVert _{L_{p}(\mathcal{M};\ell _{\infty }^{c})}=\inf \bigl\{
\sup _{n\ge 0} \llVert y_{n} \rrVert _{\infty }
\llVert a \rrVert _{p} \bigr\},
\end{equation*}
where the infimum runs over all factorizations as above. Similarly, we
denote by $L_{p}(\mathcal{M};c_{0}^{c})$ the closure of finite sequences
in $L_{p}(\mathcal{M};\ell _{\infty }^{c})$. (We refer to
\cite{defantjunge04maxau} and \cite{musat03interpolationbmolp} for more
information.)

Indeed, for any index $I$, we can define the spaces
$L_{p}(\mathcal{M};\ell _{\infty }(I))$ of families
$(x_{i})_{i\in I}$ in $L_{p}(\mathcal{M})$ with similar factorizations
as above. We omit the details and we will simply denote the spaces by the
same notation $L_{p}(\mathcal{M};\ell _{\infty })$ and
$L_{p}(\mathcal{M};c_{0})$ if no confusion can occur.

The following properties will be useful here.
\begin{prop}%
	\label{prop:nc max funct}
	\mbox{}%
	\begin{enumerate}
		\item[(1)] A~family
		$(x_{i})_{i\in I}\subset L_{p}(\mathcal{M})$ belongs to
		$L_{p}(\mathcal{M};\ell _{\infty })$ if and only if\vadjust{\goodbreak}
		\begin{equation*}
		\sup _{J\mathrm{ finite}} \bigl\llVert \mathop{{
				\sup _{i\in J}}^{+}}x_{i} \bigr\rrVert
		_{p} < \infty ,
		\end{equation*}
		and in this case
		\begin{equation*}
		\bigl\llVert \mathop{{
				\sup _{i\in I}}^{+}}x_{i} \bigr\rrVert
		_{p}=
		\sup _{J
			\mathrm{ finite}} \bigl\llVert \mathop{{
				\sup _{i\in J}}^{+}}x_{i} \bigr\rrVert
		_{p}.
		\end{equation*}
		\item[(2)] Let $1\le p_{0}<p_{1}\le \infty $, and let $0<\theta <1$. Then
		we have isometrically
		\begin{equation}
		L_{p}(\mathcal{M};\ell _{\infty })= \bigl(L_{p_{0}}(
		\mathcal{M};\ell _{\infty }), L_{p_{1}}( \mathcal{M};\ell
		_{\infty }) \bigr)_{\theta } \label{eq:interp-vectorLp} ,
		\end{equation}
		where $\frac{1}{p}=\frac{1-\theta }{p_{0}}+\frac{\theta }{p_{1}}$. If additionally
		$p_{0}\geq 2$, then we have isometrically
		\begin{equation}
		\label{eq:c-interp-vectorLp} L_{p}(\mathcal{M};\ell _{\infty }^{c})=
		\bigl(L_{p_{0}}( \mathcal{M};\ell _{\infty }^{c}),
		L_{p_{1}}( \mathcal{M};\ell _{\infty }^{c})
		\bigr)_{\theta },
		\end{equation}
		where $\frac{1}{p}=\frac{1-\theta }{p_{0}}+\frac{\theta }{p_{1}}$.
	\end{enumerate}
\end{prop}
Based on these notions we can discuss the noncommutative maximal inequalities.
\begin{defn}%
	\label{defn:type-ineq}
	Let $1\leq p\leq \infty $, and let $S=(S_{i})_{i\in I}$ be a family of
	maps from $L_{p}^{+}(\mathcal{M})$ to
	$L_{0}^{+}(\mathcal{M})$.
	\begin{enumerate}
		\item[(1)] For $p<\infty $, we say that $S$ is of
		\emph{weak type $(p,p)$ with constant $C$} if there exists a constant
		$C>0$ such that, for all $x\in L_{p}^{+}(\mathcal{M})$ and
		$\lambda >0$, there is a projection $e\in \mathcal{M}$ satisfying
		\begin{equation*}
		\tau (1-e)\leq
		\frac{C^{p}}{\lambda ^{p}} \llVert x \rrVert _{p}^{p},\quad
		eS_{i}(x)e \leq \lambda e,i\in I.
		\end{equation*}
		\item[(2)] For $1\leq p\leq \infty $, we say that $S$ is of \emph{strong
			type $(p,p)$ with constant $C$} if there exists a constant $C>0$ such that
		\begin{equation*}
		\bigl\llVert {
			\sup _{i\in I}}^{+}S_{i}x \bigr\rrVert
		_{p}\leq C \llVert x \rrVert _{p},\quad x\in L_{p}(
		\mathcal{M}).
		\end{equation*}
	\end{enumerate}
\end{defn}
We will also need a reduction below for weak type inequalities.
\begin{lem}
	\label{lem:reduction finite weak}%
	If for all finite subsets $J\subset I$, $(S_{i})_{i\in J}$ is of weak type
	$(p,p)$ with constant $C$, then $(S_{i})_{i\in I}$ is of weak type
	$(p,p)$ with constant $2^{1/p}C$.
\end{lem}
\begin{proof}
	Indeed, the proof of \cite[Lemma~3.2]{hong17dimfree} shows the following
	general property: if a family
	$(x_{i})_{i\in I}\subset L_{0}^{+}(\mathcal{M})$ and a constant
	$c>0$ satisfy that for all $\lambda >0$ and all finite subsets
	$J\subset I$, there exists a projection
	$e_{J}\in \mathcal{M}$ with
	\begin{equation*}
	\tau (1-e_{J})\leq c^{p} \lambda
	^{-p},\quad\quad \llVert e_{J} x_{i}
	e_{J} \rrVert _{\infty }\leq \lambda ,\quad i\in J,
	\end{equation*}
	then there exists a projection $e \in \mathcal{M}$ with
	\begin{equation*}
	\tau (1-e )\leq 2 c ^{p} \lambda ^{-p},
	\quad\quad \llVert e x_{i} e \rrVert _{\infty
	}\leq \lambda ,\quad
	i\in I.
	\end{equation*}
	To see this, we note that $\mathcal{S} _{+}$ is dense in
	$L_{0}^{+}(\mathcal{M})$ with respect to the topology of convergence
	in measure. So we may assume without loss of generality that
	$(x_{i})_{i\in I}\subset \mathcal{S} _{+}$. Then it suffices to take a w*-accumulation
	point $u $ of $(e_{J})_{J}$ and set the spectral projection
	$e=\chi _{[ \frac{1}{2}, 1]}(u)$. In this case we have
	\begin{align*}
	\llVert ex_{i} e \rrVert _{\infty }&= \llVert eux_{i}
	ue \rrVert _{\infty
	}
	\\
	&\leq \llVert ux_{i} u \rrVert _{\infty }= \llVert
	x_{i}^{1/2} u \rrVert _{\infty }^{2}
	=
	\sup _{y
		\in L_{1} (\mathcal{M}) \cap \mathcal{M}}
	\frac{ \llvert  \tau (x_{i}^{1/2} u y) \rrvert  ^{2}}{ \llVert  y \rrVert  _{1}^{2}}
	\\
	&=
	\sup _{y\in L_{1} (
		\mathcal{M})}
	\sup _{J}
	\frac{ \llvert  \tau (x_{i}^{1/2} e_{J} y) \rrvert  ^{2}}{ \llVert  y \rrVert  _{1}^{2}}
	\\
	&\leq
	\sup _{J} \llVert x_{i}^{1/2}
	e_{J} \rrVert _{\infty }^{2} \leq \lambda
	\end{align*}
	and
	\begin{equation*}
	e = \chi _{[ \frac{1}{2}, 1]}(u)\leq 2u, \quad 1-e \leq 2(1-u) .
	\end{equation*}
	Then we obtain the claimed property. The lemma follows by taking
	$x_{i} = S_{i}(x )$ and $c=C \llVert  x \rrVert  _{p}$.
\end{proof}
The following noncommutative Doob inequalities will play a crucial role
in our proof.
\begin{lem}[{\cite[Proposition~6]{cuculescu71doob},\cite[Theorem~0.2]{junge02doob}}]
	\label{lem:doob}%
	Let $(\mathcal{M}_{n})_{n\in \mathbb{Z}}$ be
	an increasing sequence of von Neumann subalgebras of $\mathcal{M}$ such
	that $\bigcup _{n\in \mathbb{Z}}\mathcal{M}_{n}$ is w*-dense in
	$\mathcal{M}$. Denote by $\mathbb{E}_{n}$ the $\tau $-preserving conditional
	expectation from $L_{p}(\mathcal{M})$ onto $L_{p}(\mathcal{M}_{n})$. Then
	$(\mathbb{E}_{n})_{n\in \mathbb{Z}}$ is of weak type $(1,1)$ with a universal
	constant and is of strong type $(p,p)$ with constants depending only on
	$p$ for all $1<p<\infty $.
\end{lem}
Note that only monotone sequences of the form
$(\mathbb{E}_{n})_{n\geq 0}$ are concerned in the statement of
\cite{cuculescu71doob} and \cite{junge02doob}. However, we can easily deduce
the general result in Lemma~\ref{lem:doob} from the previous case. Indeed,
for a sequence of the form $(\mathbb{E}_{n})_{n\in \mathbb{Z}}$ in the
above lemma, the statement in \cite{cuculescu71doob} and
\cite{junge02doob} yields that for all $k\in \mathbb{Z}$, the sequence
$(\mathbb{E}_{n})_{ n \geq k}$ is of weak type $(1,1)$ and strong type
$(p,p)$ with constants independent of $k $ by reindexing the sequence.
In particular, $(\mathbb{E}_{n})_{ -k \leq n \leq k}$ satisfies the same
maximal inequalities. Then by Proposition~\ref{prop:nc max funct}(1) and
Lemma~\ref{lem:reduction finite weak}, we see that the same property holds
for the sequence $(\mathbb{E}_{n})_{n\in \mathbb{Z}}$ as well.

\subsection{Actions by amenable groups}
\label{sub:Actions-by-amenable}

Unless explicitly stated otherwise, throughout $G$ will denote a locally
compact group with neutral element $e$, equipped with a fixed right invariant
Haar measure~$m$. For a Banach space $E$, we say that
\begin{equation*}
\alpha :G\to B(E),\quad\quad s\mapsto \alpha _{s}
\end{equation*}
is an action if $\alpha _{s}\circ \alpha _{h}=\alpha _{st}$ for all
$s,t\in G$. Let $(\mathcal{M},\tau )$ be as before. For a fixed
$1\leq p\leq \infty $, we will be interested in actions
$\alpha =(\alpha _{s})_{s\in G}$ on $L_{p}(\mathcal{M})$ with
the following conditions:
\begin{enumerate}
	\item[$(\mathbf{A} ^{p}_{1})$] Continuity: for all
	$x\in L_{p}(\mathcal{M})$, the map $s\mapsto \alpha _{s}x$ from $G$ to
	$L_{p}(\mathcal{M})$ is continuous. Here we take the norm topology on
	$L_{p}(\mathcal{M})$ if $1\leq p<\infty $ and the w{*}-topology if
	$p=\infty $.
	\item[$(\mathbf{A} ^{p}_{2})$] Uniform boundedness:
	$\sup _{s\in G} \llVert  \alpha _{s}:L_{p}(\mathcal{M})\to L_{p}(
	\mathcal{M}) \rrVert  <\infty $.
	\item[$(\mathbf{A} ^{p}_{3})$] Positivity: for all $s\in G$,
	$\alpha _{s}x\geq 0$ if $x\geq 0$ in $L_{p}(\mathcal{M})$.
\end{enumerate}
As a natural example, if $\alpha $ is an action on $\mathcal{M}$ satisfying
the condition:
\begin{enumerate}
	\item[$(\mathbf{A} ')$] for all $x\in \mathcal{M}$, the map
	$s\mapsto \alpha _{s}x$ from $G$ to $\mathcal{M}$ is continuous with respect
	to the w{*}-topology on $\mathcal{M}$; and for all $s\in G$,
	$\alpha _{s}$ is an automorphism of $\mathcal{M}$ (in the sense of
	$*$-algebraic structures) such that $\tau =\tau \circ \alpha _{s}$,
\end{enumerate}
then $\alpha $ extends naturally to actions on $L_{p}(\mathcal{M})$ with
conditions $(\mathbf{A} ^{p}_{1})$--$(\mathbf{A} ^{p}_{3})$ for all
$1\leq p\leq \infty $, still denoted by $\alpha $ (see, e.g.,
\cite[Lemma~1.1]{jungexu07erg}). In this case for each $s\in G$,
$\alpha _{s}$ is an isometry on $L_{p}(\mathcal{M})$. We refer to
\cite{bekka15kazhdannclp,olivier12kazhdannclp}, and
\cite{olivier13thesisactionlp} for other natural examples of group actions
on noncommutative $L_{p}$-spaces.

Recall that $G$ is said to be \emph{amenable} if $G$ admits a
\emph{F{\o}lner net}, that is, a net $(F_{i})_{i\in I}$ of measurable subsets
of $G$ with $m(F_{i})<\infty $ such that for all $s\in G$,
\begin{equation*}
\lim _{i}
\frac{m((F_{i}s)\bigtriangleup F_{i})}{m(F_{i})}=0.
\end{equation*}
Note that the above condition is a reformulation of the asymptotic invariance
\eqref{eq:asymp inv} for the general setting. It is known that
$(F_{i})_{i\in I}$ is a F{\o}lner net if for all compact measurable subsets
$K\subset G$,
\begin{equation}
\lim _{i}
\frac{m(F_{i}K)}{m(F_{i})}=1. \label{eq:folner condition cpt}
\end{equation}
Recall that $G$ is a compactly generated group
\emph{of polynomial growth} if the compact generating subset
$V\subset G$ satisfies
\begin{equation*}
m(V^{n})\leq kn^{r},\quad n\geq 1,
\end{equation*}
where $k>0$ and $r\in \mathbb{N}$ are constants independent of~$n$. It
is well known that any group of polynomial growth is amenable and the sequence
$(V^{n})_{n\geq 1}$ satisfies the above F{\o}lner condition (see, e.g.,
\cite{breuillard14polygrowth,tessera07polygrowth}). We refer to
\cite{paterson88amenability} for more information on amenable groups.

Now let $G$ be amenable, and let $(F_{i})_{i\in I}$ be a F{\o}lner net
in~$G$. Let $1<p<\infty $. Let $\alpha =(\alpha _{s})_{s\in G}$ be an action
of $G$ on $L_{p}(\mathcal{M})$ satisfying
$(\mathbf{A} ^{p}_{1})$--$(\mathbf{A} ^{p}_{3})$. Denote by $A_{i}$ the corresponding
averaging operators
\begin{equation*}
A_{i}x=
\frac{1}{m(F_{i})}
\int _{F_{i}}\alpha _{s}x\,dm(s),\quad x\in
L_{p}( \mathcal{M}).
\end{equation*}
According to the mean ergodic theorem for amenable groups (see, e.g.,
\cite[Th\'{e}or\`{e}me 2.2.7]{orlean10ergbook}), we have a canonical splitting
on $L_{p}(\mathcal{M})$:
\begin{equation*}
L_{p}(\mathcal{M})=\mathcal{F}_{p}\oplus
\mathcal{F}_{p}{}^{
	\perp }
\end{equation*}
with
\begin{equation}
\begin{aligned}
\mathcal{F}_{p}&= \bigl\{x\in L_{p}(\mathcal{M}): \alpha
_{s}x=x,s \in G \bigr\}, \\\mathcal{F}_{p}{}^{\perp }&=
\overline{\mathrm{span}} \bigl\{x- \alpha _{s}x:s\in G,x\in
L_{p}(\mathcal{M}) \bigr\}. \label{eq:decomposition erg}
\end{aligned}
\end{equation}
Let $P$ be the bounded positive projection from
$L_{p}(\mathcal{M})$ onto $\mathcal{F}_{p}$. Then
$(A_{i}x)_{i}$ converges to $Px$ in $L_{p}(\mathcal{M})$ for
all $x\in L_{p}(\mathcal{M})$.

Assume additionally that $\alpha $ extends to an action on
$\bigcup _{1\leq p\leq \infty }L_{p}(\mathcal{M})$ satisfying
$(\mathbf{A} ^{p}_{1})$--$(\mathbf{A} ^{p}_{3})$ for \emph{every}
$1\leq p\leq \infty $. Note that the convergence in
$L_{p}(\mathcal{M})$ yields the convergence in measure in
$L_{0}(\mathcal{M})$, and in particular for
$x\in L_{1}^{+}(\mathcal{M})\cap \mathcal{M}_{+} $ and for $p_{0}=1$ or
$p_{0}=\infty $,
\begin{equation*}
\llVert Px \rrVert _{p_{0}}\leq
\liminf _{n\to \infty } \llVert A_{n}x \rrVert _{p_{0}}
\leq
\sup _{s
	\in G} \llVert \alpha _{s} \rrVert
_{B(L_{p_{0}}(\mathcal{M}))} \llVert x \rrVert _{p_{0}},
\end{equation*}
so by \cite[Lemma~1.1]{jungexu07erg} and
\cite[Proposition~1]{yeadon77max}, $P$ admits a continuous extension on
$L_{1}(\mathcal{M})$ and $\mathcal{M}$, still denoted
by~$P$. The splitting \eqref{eq:decomposition erg} is also true in this
case. Note then, however, that $\mathcal{F}_{\infty }{}^{\perp }$ is the
w*-closure of the space spanned by
$\{x-\alpha _{s}x:s\in G,x\in \mathcal{M}\}$.

\section{Noncommutative Calder\'on's transference principle}
\label{sect:transference}

In this section, we discuss a noncommutative variant of Calder\'{o}n's
transference principle. Fix $1\leq p< \infty $. Let $G$ be a locally compact
group,\vadjust{\goodbreak}  and let $\alpha $ be an action satisfying
$(\mathbf{A} ^{p}_{1})$--$(\mathbf{A} ^{p}_{3})$ in the previous section.
Let $(\mu _{n})_{n\geq 1}$ be a sequence of Radon probability measures
on~$G$. We consider the following averages:
\begin{equation}
A_{n}x=
\int _{G}\alpha _{s}x\,d\mu _{n}(s),\quad
x \in L_{p}( \mathcal{M}),n\geq 1. \label{eq:average mu_n}
\end{equation}
Let us also consider the natural translation action of $G$ on itself. We
are interested in the following averages: for all
$f\in L_{p}(G;L_{p}(\mathcal{M}))$,
\begin{equation}
A_{n}'f(s)=
\int _{G}f(st)\,d\mu _{n}(t),\quad s\in G,n\geq 1,
\label{eq:op conv mu_n}
\end{equation}
where the integration denotes the usual integration of Banach space-valued
functions.

\subsection{Strong type inequalities}
\label{sec3.1}

We begin with the transference principle for strong type $(p,p)$ inequalities.
\begin{thm}
	\label{thm:transference}%
	Assume that $G$ is amenable. Fix $1\leq p<\infty $. If there exists a constant
	$C>0$ such that
	\begin{equation*}
	\bigl\llVert {
		\sup _{n}}^{+}A_{n}'f \bigr\rrVert
	_{p}\leq C \llVert f \rrVert _{p},\quad f\in L_{p}
	\bigl(G;L_{p}( \mathcal{M}) \bigr),
	\end{equation*}
	then there exists a constant $C'>0$ depending on $\alpha $ such that
	\begin{equation*}
	\bigl\llVert {
		\sup _{n}}^{+}A_{n}x \bigr\rrVert
	_{p}\leq CC' \llVert x \rrVert _{p},\quad x\in
	L_{p}( \mathcal{M}).
	\end{equation*}
\end{thm}

\begin{proof}
	Note that we may take an increasing net of compact subsets
	$K_{i}\subset G$ such that $\lim _{i}\mu _{n}(K_{i})=\mu _{n} (G)$. Then
	for
	\begin{equation*}
	A_{n,i}x=
	\int _{G}\alpha _{s}x \chi _{K_{i}}(s) \,d
	\mu _{n}(s),\quad x \in L_{p}(\mathcal{M}),n\geq 1,
	\end{equation*}
	we have
	\begin{equation*}
	A_{n,i}x\leq A_{n} x,\quad\quad
	\lim _{i\to \infty } \llVert A_{n} x - A_{n,i}x
	\rrVert _{p}=0, \quad x\in L_{p}^{+}(
	\mathcal{M}).
	\end{equation*}
	So for $x,y\in L_{p}^{+}(\mathcal{M})$,
	\begin{equation*}
	A_{n} x\leq y \quad \quad \text{iff}\quad \quad \forall i , \quad A_{n,i}x\leq
	y.
	\end{equation*}
	Hence
	$ \llVert  \sup _{n,i}^{+}A_{n,i}x_{n,i} \rrVert  _{p} = \llVert  \sup _{n}^{+}A_{n}x_{n} \rrVert  _{p}$.
	So without loss of generality we may assume that the $\mu _{n}$'s are of
	compact support.
	
	We fix $x\in L_{p}(\mathcal{M})$ and $N\geq 1$. Choose a compact
	subset $K\subset G$ such that $\mu _{n}$ is supported in $K$ for all
	$1\leq n\leq N$. Since
	$\alpha _{s}:L_{p}(\mathcal{M})\to L_{p}(
	\mathcal{M})$ is positive for all $s\in G$, we see that
	$(\alpha _{s}\otimes \mathrm{Id})_{s\in G}$ extends to a uniformly bounded
	family of maps on $L_{p}(\mathcal{M};\ell _{\infty })$ (see, e.g.,
	\cite[Proposition~7.3]{haagerupjx10reduction}). So we may choose a constant
	$C'>0$ such that
	\begin{equation*}
	\llVert \mathop{{\sup }^{+}}_{1\le n\le N}A_{n}x
	\rrVert _{p}= \llVert \mathop{{ \sup }^{+}}_{1\le n\le N}
	\alpha _{s^{-1}}\alpha _{s}A_{n}x \rrVert
	_{p} \leq C' \llVert \mathop{{\sup }^{+}}_{1\le n\le N}
	\alpha _{s}A_{n}x \rrVert _{p}, \quad s\in G.
	\end{equation*}
	Let $F$ be a compact subset. Then we have
	\begin{equation}
	\llVert \mathop{{\sup }^{+}}_{1\le n\le N}A_{n}x
	\rrVert _{p}^{p}\leq C^{\prime p}
	\frac{1}{m(F)}
	\int _{F} \llVert \mathop{{\sup }^{+}}_{1\le n\le N}
	\alpha _{s}A_{n}x \rrVert _{p}^{p}
	\,dm(s). \label{eq:alpha_g An p p}
	\end{equation}
	We define a function $f\in L_{p}(G;L_{p}(\mathcal{M}))$ as
	\begin{equation*}
	f(h)=\chi _{FK}(h)\alpha _{h}x,\quad h\in G.
	\end{equation*}
	Then for all $s\in F$,
	\begin{equation}
	\alpha _{s}A_{n}x=
	\int _{K}\alpha _{st}x\,d\mu _{n}(t)=
	\int _{K}f(st)\,d\mu _{n}(t)=A_{n}'f(s).
	\label{eq:transf alpha_g An}
	\end{equation}
	We consider
	$(A_{n}'f)_{1\leq n\leq N}\in L_{p}(L_{\infty }(G)\bar{\otimes }
	\mathcal{M};\ell _{\infty })$, and for any $\varepsilon >0$ we
	take a factorization $A_{n}'f=aF_{n}b$ such that
	$a,b\in L_{2p}(L_{\infty }(G)\bar{\otimes }\mathcal{M})$,
	$F_{n}\in L_{\infty }(G)\bar{\otimes }\mathcal{M}$, and
	\begin{equation*}
	\llVert a \rrVert _{2p}
	\sup _{1\leq n\leq N} \llVert F_{n} \rrVert _{\infty }
	\llVert b \rrVert _{2p}\leq \bigl\llVert (A_{n}'f)_{1\leq n\leq N}
	\bigr\rrVert _{L_{p}(L_{\infty }(G)
		\bar{\otimes }\mathcal{M};\ell _{\infty })}+\varepsilon .
	\end{equation*}
	Then we have
	\begin{align*}
	\int _{G} \bigl\llVert \mathop{{\sup }^{+}}_{1\le n\le N}A_{n}'f(s)
	\bigr\rrVert _{p}^{p}\,dm(s) &\leq
	\int _{G} \bigl\llVert a(s) \bigr\rrVert _{2p}^{p}
	\sup _{1\leq n\leq N} \bigl\llVert F_{n}(s) \bigr\rrVert
	_{
		\infty }^{p} \bigl\llVert b(s) \bigr\rrVert
	_{2p}^{p}\,dm(s)
	\\
	&\leq \llVert a \rrVert _{2p}^{p}
	\sup _{1\leq n\leq N} \llVert F_{n} \rrVert _{\infty }^{p}
	\llVert b \rrVert _{2p}^{p}
	\\
	&\leq \bigl( \bigl\llVert
	(A_{n}'f)_{1\leq n\leq N} \bigr\rrVert _{L_{p}(L_{\infty }(G)
		\bar{\otimes }\mathcal{M};\ell _{\infty })}+\varepsilon
	\bigr)^{p}.
	\end{align*}
	Since $\varepsilon $ is arbitrarily chosen, we obtain
	\begin{equation*}
	\int _{G} \bigl\llVert \mathop{{\sup }^{+}}_{1\le n\le N}A_{n}'f(s)
	\bigr\rrVert _{p}^{p}\,dm(s) \leq \llVert \mathop{{
			\sup }^{+}}_{1\le n\le N}A_{n}'f \rrVert
	_{p}^{p}.
	\end{equation*}
	Thus, together with \eqref{eq:alpha_g An p p},
	\eqref{eq:transf alpha_g An}, and the assumption, we see that
	\begin{align*}
	\llVert \mathop{{\sup }^{+}}_{1\le n\le N}A_{n}x
	\rrVert _{p}^{p} &\leq
	\frac{C^{\prime p}}{m(F)}
	\int _{F} \bigl\llVert \mathop{{\sup }^{+}}_{1\le n\le N}A_{n}'f(s)
	\bigr\rrVert _{p}^{p}\,dm(s)\leq
	\frac{C^{\prime p}}{m(F)} \llVert \mathop{{\sup }^{+}}_{1
		\le n\le N}A_{n}'f
	\rrVert _{p}^{p}
	\\
	&\leq
	\frac{C^{p}C^{\prime p}}{m(F)} \llVert f \rrVert _{p}^{p}=
	\frac{C^{p}C^{\prime p}}{m(F)}
	\int _{FK} \llVert \alpha _{h}x \rrVert
	_{p}^{p}\,dm(h)
	\\
	&\leq
	\frac{C^{p}C^{\prime p}m(FK)}{m(F)} \llVert x \rrVert _{p}^{p}.
	\end{align*}
	Since $G$ is amenable, for any $\varepsilon >0$ we may choose the above
	subset $F$ such that $m(FK)/m(F)\leq 1+\varepsilon $. Therefore, we obtain
	\begin{equation*}
	\llVert \mathop{{\sup }^{+}}_{1\le n\le N}A_{n}x
	\rrVert _{p}\leq CC'(1+ \varepsilon ) \llVert x \rrVert _{p}.
	\end{equation*}
	Note that $N$, $\varepsilon $, $x$ are all arbitrarily chosen, so we establish
	the theorem.%
\end{proof}
\begin{rem}
	\label{rem:rk strong transf}%
	Applying the same argument, we may obtain several variants of the above
	theorem.
	\begin{enumerate}
		\item[(1)] The sequence of measures $(\mu _{n})_{n\geq 1}$ can be replaced
		by any family $(\mu _{i})_{i\in I}$ of Radon probability measures for an
		arbitrary index set~$I$.
		\item[(2)] The positivity of the action $\alpha $ can be replaced by more
		general assumptions. It suffices to assume that
		\begin{equation*}
		\sup _{s\in G} \llVert \alpha _{s}\otimes \mathrm{Id}
		\rrVert _{B(L_{p}(\mathcal{M};
			\ell _{\infty }))}<\infty .
		\end{equation*}
		If $\mathcal{M}$ is commutative, then this is equivalent to saying
		that the operators $(\alpha _{s})_{s\in G}$ are regular with uniformly
		bounded regular norm (see \cite{meyernieberg91banachlattices}). In the
		noncommutative setting, one may assume that $(\alpha _{s})_{s\in G}$ are
		uniformly bounded decomposable maps, and we refer to
		\cite{jungeruan04decomposable} and \cite{pisier95regular} for more details.
		\item[(3)] One may also state similar properties for transference of linear
		operators; in this case the assumption on positivity of $\alpha $ can be
		ignored, and the semigroup actions can be included. We have the following
		noncommutative analogue of the transference result in
		\cite[Theorem~2.4]{coifmanweiss76transference}. Assume that $G$ and
		$\alpha $ satisfy one of the following conditions:
		\begin{enumerate}
			\item[(a)] $G$ is an amenable locally compact group, and $\alpha $ satisfies
			$(\mathbf{A} ^{p}_{1})$ and $(\mathbf{A} ^{p}_{2})$;
			\item[(b)] $G$ is a discrete amenable semigroup or
			$G=\mathbb{R}_{+}$, $\alpha $ satisfies $(\mathbf{A} ^{p}_{1})$, and each
			$\alpha _{s}$ is an isometry on $L_{p}(\mathcal{M})$ (or more generally,
			there exist $K_{1},K_{2}>0$ such that for all $s\in G$, we have
			$K_{1} \llVert  x \rrVert  _{p}\leq  \llVert  \alpha _{s}x \rrVert  _{p}\leq K_{2} \llVert  x \rrVert  _{p}$).
		\end{enumerate}
		Let $\mu $ be a bounded Radon measure on~$G$. Define
		\begin{equation*}
		T_{\mu }(f) (s)=
		\int _{G}f(st)\,d\mu (t),\quad f\in L_{p}(G), s\in
		G,
		\end{equation*}
		and
		\begin{equation*}
		\widetilde{T}_{\mu }(x)=
		\int _{G}\alpha _{s}x\,d\mu (s),\quad x\in
		L_{p}( \mathcal{M}).
		\end{equation*}
		Then we have
		\begin{equation*}
		\llVert \widetilde{T}_{\mu } \rrVert _{B(L_{p}(\mathcal{M}))}\leq
		\sup _{s
			\in G} \llVert \alpha _{s} \rrVert
		_{B(L_{p}(\mathcal{M}))} \llVert T_{\mu } \otimes \mathrm{Id} \rrVert
		_{B(L_{p}(L_{\infty }(G)\bar{\otimes }
			\mathcal{M}))}.
		\end{equation*}
	\end{enumerate}
\end{rem}

\subsection{Weak type inequalities}
\label{sec3.2}

Now we discuss the transference principle for weak type $(p,p)$ inequalities.
In this case we will only consider the special case of group actions on
von Neumann algebras. We assume that $\alpha $ is given by an action on
$\mathcal{M}$ satisfying the condition $(\mathbf{A} ')$ in Section~\ref{sub:Actions-by-amenable}.
\begin{thm}
	\label{thm:transference weak}%
	Assume that $G$ is amenable. Let $(A_{n})_{n\geq 1}$ and
	$(A_{n}')_{n\geq 1}$ be the associated sequences of maps given in
	\eqref{eq:average mu_n} and \eqref{eq:op conv mu_n}. Fix
	$1\leq p<\infty $. If the sequence $(A_{n}')_{n\geq 1}$ is of weak type
	$(p,p)$, then $(A_{n})_{n\geq 1}$ is of weak type $(p,p)$ too.%
\end{thm}
\begin{proof}
	As in the last subsection, we may assume without loss of generality that
	$\mu _{n}$ is of compact support. Assume that the sequence
	$(A_{n}')_{n\geq 1}$ is of weak type $(p,p)$ with constant~$C$. By Lemma~\ref{lem:reduction finite weak}, it suffices to show that there exists
	a constant $C'>0$ such that for all
	$\lambda >0$, $x\in L_{p}^{+}(\mathcal{M})$, $N\geq 1$, there exists
	a projection $e\in \mathcal{M}$ such that
	\begin{equation*}
	\tau (e^{\bot })\leq
	\frac{C^{\prime p}}{\lambda ^{p}} \llVert x \rrVert _{p}^{p},\quad
	\quad e(A_{n}x)e \leq \lambda ,\quad 1\leq n\leq N.
	\end{equation*}
	We fix $\lambda >0$, $x\in L_{p}(\mathcal{M})$, and $N\geq 1$. Choose
	a compact subset $K\subset G$ such that $\mu _{n}$ is supported in
	$K$ for all $1\leq n\leq N$. Let $F$ be a compact subset. We define a function
	$f\in L_{p}(G;L_{p}(\mathcal{M}))$ as
	\begin{equation*}
	f(t)=\chi _{FK}(t)\alpha _{t}x,\quad t\in G.
	\end{equation*}
	Then for all $s\in F$,
	\begin{equation}
	\alpha _{s}A_{n}x=
	\int _{K}\alpha _{st}x\,d\mu _{n}(t)=
	\int _{K}f(st)\,d\mu _{n}(t)=A_{n}'f(s).
	\label{eq:transf alpha_g An-1}
	\end{equation}
	Since the sequence $(A_{n}')_{n\geq 1}$ is of weak type $(p,p)$ with constant
	$C$, we may choose a projection
	$e\in L_{\infty }(G)\bar{\otimes }\mathcal{M}$ such that
	\begin{equation*}
	\int _{G}\tau \bigl(e(s)^{\bot } \bigr)\,ds
	\leq \Bigl(C
	\frac{ \llVert  f \rrVert  _{L_{p}(G;L_{p}(\mathcal{M}))}}{\lambda } \Bigr)^{p} \quad\quad \text{and}\quad\quad
	e(A_{n}'f)e\leq \lambda ,\quad 1\leq n\leq N,
	\end{equation*}
	where we regard $e$ as a measurable operator-valued function on~$G$. Therefore,
	by \eqref{eq:transf alpha_g An-1} for $s\in G$, we have
	\begin{equation}
	\bigl(\alpha _{s^{-1}}e(s) \bigr) (A_{n}x) \bigl(\alpha
	_{s^{-1}}e(s) \bigr)\leq \lambda , \quad n\geq 1. \label{eq:a.s. lambda bdd}
	\end{equation}
	Recall that each $\alpha _{s^{-1}}$ is a unital $\tau $-preserving automorphism
	of $\mathcal{M}$. In particular, for an arbitrary
	$\varepsilon >0$, we may choose $s_{0}\in G$ and a projection
	$\tilde{e}\coloneqq \alpha _{s_{0}^{-1}}e(s_{0})\in
	\mathcal{M}$ such that
	\begin{equation*}
	\tau (\tilde{e}^{\bot })\leq
	\inf _{s\in G}\tau \bigl(e(s)^{\bot } \bigr)+
	\varepsilon \quad\quad \text{and}\quad\quad \tilde{e}(A_{n}x)\tilde{e}
	\leq \lambda ,\quad n\geq 1.
	\end{equation*}
	Then we have
	\begin{align*}
	\tau (\tilde{e}^{\bot }) &\leq
	\frac{1}{m(F)}
	\int _{F}\tau \bigl(e(s)^{
		\bot } \bigr)\,ds+
	\varepsilon \leq
	\frac{C^{p}}{\lambda ^{p}m(F)} \llVert f \rrVert _{L_{p}(G;L_{p}(
		\mathcal{M}))}^{p}+
	\varepsilon
	\\
	&=
	\frac{C^{p}}{\lambda ^{p}m(F)}
	\int _{FK} \llVert \alpha _{h}x \rrVert
	_{p}^{p}\,dm(h)+ \varepsilon
	\\
	&\leq
	\frac{C^{p}m(FK)}{\lambda ^{p}m(F)} \llVert x \rrVert _{p}^{p}+
	\varepsilon .
	\end{align*}
	Since $G$ is amenable, for any $\varepsilon >0$ we may choose the above
	subset $F$ such that $m(FK)/m(F)\leq 1+\varepsilon $. Therefore, we obtain
	\begin{equation*}
	\tau (\tilde{e}^{\bot })\leq
	\frac{C^{p}(1+\varepsilon )}{\lambda ^{p}} \llVert x \rrVert _{p}^{p}+
	\varepsilon , \quad\quad \tilde{e}(A_{n}x)\tilde{e}\leq \lambda ,\quad 1
	\leq n\leq N.
	\end{equation*}
	Note that $N$, $\varepsilon $, $x$ are all arbitrarily chosen, so we establish
	the theorem.%
\end{proof}
\begin{rem}
	We remark that the above result also holds for other index sets. For example,
	$(\mu _{n})_{n\geq 1}$ can be replaced by a one-parameter family
	$(\mu _{r})_{r\in \mathbb{R}_{+}}$ such that $r\mapsto A_{r}x$ is continuous
	for all $x\in L_{p}(\mathcal{M})$. The only ingredient needed
	in the proof is that the condition \eqref{eq:a.s. lambda bdd} hold true
	almost everywhere on~$G$.
\end{rem}

\section{Maximal inequalities: Probabilistic approach}
\label{sect:max}

This section is devoted to the proof of the maximal inequalities in Theorem~\ref{thm:main intro}. To this end, we will first establish noncommutative
Hardy--Littlewood maximal inequalities on doubling metric spaces.

\subsection{Hardy--Littlewood maximal inequalities on metric measure spaces}
\label{sec4.1}

Throughout, a metric measure space $(X,d,\mu )$ refers to a metric space
$(X,d)$ equipped with a Radon measure~$\mu$. We denote
$B(x,r)=\{y\in X:d(y,x)\leq r\}$, and we say that $\mu $ satisfies the
\emph{doubling condition} if there exists a constant $K>0$ such that
\begin{equation}
\label{eq:measure doubling} \forall r>0,x\in X,\quad \mu \bigl(B(x,2r) \bigr)\leq K\mu
\bigl(B(x,r) \bigr).
\end{equation}
In the sequel, we always assume the nondegeneracy property
$0<\mu (B(x,r))<\infty $ for all $r>0$. The following theorem can be regarded
as an operator-valued analogue of the Hardy--Littlewood maximal inequalities.
\begin{thm}
	\label{thm:hl max}%
	Let $(X,d,\mu )$ be a metric measure space. Suppose that $\mu $ satisfies
	the doubling condition. Let $1\leq p<\infty $, and let $A_{r}$ be the averaging
	operators
	\begin{equation*}
	A_{r}f(x)=
	\frac{1}{\mu (B(x,r))}
	\int _{B(x,r)}f\,d\mu ,\quad f\in L_{p}
	\bigl(X;L_{p}( \mathcal{M}) \bigr),x\in X,r>0.
	\end{equation*}
	Then $(A_{r})_{r\in \mathbb{R}_{+}}$ is of weak type $(1,1)$ and of strong
	type $(p,p)$ for $1<p<\infty $.
\end{thm}

The key ingredient of the proof is the following construction of dyadic
systems of metric measure spaces, which is established in
\cite[Corollary~7.4]{hytonenkairema12dyadic}.
\begin{lem}
	\label{lem:existence non random partition}%
	Let $(X,d)$ be a metric space, and let $\mu $ be a Radon measure on
	$X$ satisfying the doubling condition. Then there exists a finite collection
	of families
	$\mathscr{P}^{1},\mathscr{P}^{2},\ldots,\mathscr{P}^{N}$, where each
	$\mathscr{P}^{i}\coloneqq (\mathscr{P}_{k}^{i})_{k\in \mathbb{Z}}$ is a
	sequence of partitions of $X$, such that the following conditions hold
	true:
	\begin{enumerate}
		\item[(1)] for each $1\leq i\leq N$ and for each $k\in \mathbb{Z}$, the
		partition $\mathscr{P}_{k+1}^{i}$ is a refinement of the partition
		$\mathscr{P}_{k}^{i}$;
		\item[(2)] there exists a constant $C>0$ such that for all $x\in X$ and
		$r>0$, there exist $1\leq i\leq N$, $k\in \mathbb{Z}$ and an element
		$Q\in \mathscr{P}_{k}^{i}$ such that
		\begin{equation*}
		B(x,r)\subset Q,\quad\quad \mu (Q)\leq C\mu \bigl(B(x,r) \bigr).
		\end{equation*}
	\end{enumerate}
\end{lem}
\begin{rem}
	The lemma dates back to the construction of dyadic systems in the case
	of $X=\mathbb{R}^{d}$, which is due to Mei \cite{mei03dyadic,mei07hardy} (see also \cite{pisier16chapterbmo}). We remark that Mei's construction also works for the
	discrete space $\mathbb{Z}^{d}$ as follows. For $0\leq i\leq d$ and
	$k\geq 0$, we set $\tilde{\mathscr{P}}_{-k}^{i}$ to be the following family
	of intervals in $\mathbb{Z}$:
	\begin{equation*}
	\tilde{\mathscr{P}}_{-k}^{i}= \bigl\{[\alpha
	_{k}^{(i)}+m2^{k},\alpha
	_{k}^{(i)}+(m+1)2^{k}) \cap
	\mathbb{Z}:m\geq 1 \bigr\},
	\end{equation*}
	where $\alpha _{k}^{(i)}=\sum _{j=0}^{k-1}2^{j}\xi _{j}^{(i)}$ modulo
	$2^{k}$, with
	\begin{equation*}
	\xi _{(d+1)n+l}^{(i)}=\delta _{i,l},\quad n\geq
	0,0\leq l\leq d.
	\end{equation*}
	And we set $\tilde{\mathscr{P}}_{k}^{i}=\tilde{\mathscr{P}}_{0}^{i}$ for
	all $k\geq 0$. Consider the usual word metric $d$ and the counting measure
	$\mu $ on $X=\mathbb{Z}^{d}$. Then the partitions
	\begin{equation*}
	\mathscr{P}_{k}^{i}\coloneqq (\tilde{
		\mathscr{P}}_{k}^{i} )^{d},\quad
	k\in \mathbb{Z},0\leq i\leq d
	\end{equation*}
	satisfy the conditions in Lemma~\ref{lem:existence non random partition}, with constant
	$C\leq 2^{3d(d+2)}$.
\end{rem}
\begin{proof}[Proof of Theorem~\ref{thm:hl max}]
	Let $\Sigma $ be the $\sigma $-algebra of Borel sets on~$X$. For
	$1\leq i\leq N$ and $k\in \mathbb{Z}$, we define
	$\Sigma _{k}^{i}\subset \Sigma $ to be the $\sigma $-subalgebra generated
	by the elements of $\mathscr{P}_{k}^{i}$. Denote by
	$\mathbb{E}_{k}^{i}$ the conditional expectation from
	$L_{\infty }(X,\Sigma ,\mu )\bar{\otimes }\mathcal{M}$ to
	$L_{\infty }(X,\Sigma _{k}^{i},\mu| _{\Sigma _{k}^{i}})\bar{\otimes }
	\mathcal{M}$. For each $x\in X$, let
	$\mathscr{P}_{k}^{i}(x)$ be the unique element of
	$\mathscr{P}_{k}^{i}$ which contains~$x$. Then we have
	\begin{equation*}
	\mathbb{E}_{k}^{i}g(x)=
	\frac{1}{\mu (\mathscr{P}_{k}^{i}(x))}
	\int _{
		\mathscr{P}_{k}^{i}(x)}g\,d\mu ,\quad x\in X,g\in L_{\infty }(X,{\Sigma
	}, \mu )\bar{\otimes }\mathcal{M}.
	\end{equation*}
	By Lemma~\ref{lem:doob}, there exists a constant $C>0$ such that for
	$\lambda >0$ and $f\in L_{1}^{+}(X;L_{1}(\mathcal{M}))$, there
	exists a projection
	$e_{i}\in L_{\infty }(X,\Sigma ,\mu )\bar{\otimes }
	\mathcal{M}$ satisfying
	\begin{equation*}
	\tau (e_{i}^{\bot })\leq
	\frac{C}{\lambda } \llVert f \rrVert _{L_{1}(X;L_{1}(
		\mathcal{M}))},\quad\quad e_{i}(
	\mathbb{E}_{k}^{i}f)e_{i}\leq \lambda ,
	\quad k\in \mathbb{Z}.
	\end{equation*}
	Take $e=\bigwedge _{i}e_{i}$ to be the infimum of
	$(e_{i})_{i=1}^{N}$, that is, the projection onto
	$\bigcap _{i}e_{i}H$ ($H$ being the Hilbert space on which
	$\mathcal{M}$ acts). Note that
	$(\bigwedge _{i}e_{i})^{\bot }\leq \sum _{i}e_{i}^{\bot }$. Then we have
	\begin{equation}
	\tau (e^{\bot })\leq
	\frac{CN}{\lambda } \llVert f \rrVert _{L_{1}(X;L_{1}(
		\mathcal{M}))},\quad\quad e(
	\mathbb{E}_{k}^{i}f)e \leq \lambda , \quad 1\leq i
	\leq N,k\in \mathbb{Z}. \label{eq:pf max martingale}
	\end{equation}
	By Lemma~\ref{lem:existence non random partition}, there exists a constant
	$C'>0$ such that for each $x\in X$ and $r>0$, there exist
	$1\leq i\leq N$ and $k\in \mathbb{Z}$ such that
	$B(x,r)\subset \mathscr{P}_{k}^{i}(x)$, $\mu (\mathscr{P}_{k}^{i}(x))
	\leq C'\mu (B(x,r))$; in particular,
	\begin{equation}
	\label{eq:dom martingale}\begin{aligned}[b]
	A_{r}f(x)&=
	\frac{1}{\mu (B(x,r))}
	\int _{B(x,r)}f\,d\mu \leq C'
	\frac{1}{\mu (\mathscr{P}_{k}^{i}(x))}
	\int _{\mathscr{P}_{k}^{i}(x)}fd \mu
	\\
	&=C'\mathbb{E}_{k}^{i}f(x).
	\end{aligned}
	\end{equation}
	Then together with \eqref{eq:pf max martingale} we see that
	\begin{equation*}
	e(A_{r}f)e\leq C'\lambda ,\quad r>0.
	\end{equation*}
	Therefore, $(A_{r})_{r>0}$ is of weak type $(1,1)$.
	
	On the other hand, for $1<p<\infty $, according to the proof of
	\eqref{eq:dom martingale}, we have for
	$f\in L_{p}^{+}(X;L_{p}(\mathcal{M}))$,
	\begin{align*}
	\bigl\llVert (A_{r} f)_{r>0} \bigr\rrVert _{L_{p}(\mathcal{M} \bar{\otimes }L_{\infty }(X);\ell _{\infty })} &
	\leq C' \Bigl\llVert \Bigl(
	\sum _{i=1}^{N}\mathbb{E}_{k}^{i}f
	\Bigr)_{k\in
		\mathbb{Z}} \Bigr\rrVert _{L_{p}(\mathcal{M} \bar{\otimes }L_{\infty }(X);\ell _{\infty })}
	\\
	&\leq C'
	\sum _{i=1}^{N} \bigl\llVert (
	\mathbb{E}_{k}^{i}f)_{k\in \mathbb{Z}} \bigr\rrVert
	_{L_{p}(
		\mathcal{M} \bar{\otimes }L_{\infty }(X);\ell _{\infty })}.
	\end{align*}
	Since each $(\mathbb{E}_{k}^{i})_{k\in \mathbb{Z}}$ on the right-hand side
	is of strong type $(p,p)$ by Lemma~\ref{lem:doob}, we see that
	$(A_{r} )_{r\in \mathbb{R} _{+}}$ is of strong type $(p,p)$, as desired.
\end{proof}

\begin{rem}
	There is another approach to random dyadic systems of metric measure spaces,
	which is proved by Naor and Tao
	\cite[Lemma~3.1]{naortao10randommartingale}. The construction is motivated
	by the study of Hardy--Littlewood maximal inequalities on large-dimensional
	doubling spaces. Their result replaces the families
	$\mathscr{P}^{1},\mathscr{P}^{2},\ldots,\mathscr{P}^{N}$ in Lemma~\ref{lem:existence non random partition} by an infinite random collection
	$(\mathscr{P}^{\omega })_{\omega \in \Omega }$ for a probability space
	$(\Omega ,P)$, and assumes a positive probability for the coverings of
	balls. In this case we may find a random family of martingales
	$\{(\mathbb{E} _{k}^{\omega })_{k\in \mathbb{Z}},\omega \in \Omega \}$ such
	that for some $k:\mathbb{R} _{+}\to \mathbb{Z}$ and for some fixed constant
	$c$, we have
	\begin{equation*}
	A_{r}f\leq c
	\int _{\Omega }\mathbb{E}_{k(r)}^{\omega }f(x)
	\,dP( \omega ), \quad f\in L_{p}^{+}
	\bigl(X;L_{p}( \mathcal{M}) \bigr).
	\end{equation*}
	This yields as well the strong type $(p,p)$ inequalities of
	$(A_{r})_{r\in \mathbb{R} _{+}}$ for $1<p<\infty $. We omit the details.
\end{rem}

\subsection{Maximal ergodic inequalities}
\label{sec4.2}

Based on the previous result, we are now ready to deduce the following
maximal ergodic theorems. We say that a metric $d$ on $G$ is
\emph{invariant} if $d(e,g)=d(h,hg)$ for all $g,h\in G$. We denote
$B_{r}=B(e,r)$ for $r>0$. As before we consider an action $\alpha $ on
$L_{p}(\mathcal{M})$ satisfying the conditions $(\mathbf{A} ^{p}_{1})$--$(
\mathbf{A} ^{p}_{3})$ for a fixed $p$ in Section~\ref{sub:Actions-by-amenable}. The following result establishes the maximal
inequalities in Theorem~\ref{thm:main intro}.
\begin{thm}
	\label{thm:max doubling}%
	Let $G$ be an amenable locally compact group, and let $d$ be an invariant
	metric on~$G$. Assume that $(G,d)$ satisfies the doubling condition
	\eqref{eq:doubling def}. Fix $1\leq p<\infty $. Let $A_{r}$ be the averaging
	operators
	\begin{equation*}
	A_{r}x=
	\frac{1}{m(B_{r})}
	\int _{B_{r}}\alpha _{s}x\,dm(s),\quad x\in
	L_{p}( \mathcal{M}),r>0.
	\end{equation*}
	Then $(A_{r})_{r>0}$ is of strong type $(p,p)$ if $1<p<\infty $. If, in
	addition, $\alpha $ satisfies the condition \emph{$(\mathbf{A} ')$}, then
	$(A_{r})_{r>0}$ is of weak type $(1,1)$.%
\end{thm}
\begin{proof}
	By Theorems~\ref{thm:transference} and~\ref{thm:transference weak} and
	the remarks following them, it suffices to prove the maximal inequalities
	for the averaging operators
	\begin{equation*}
	A_{r}'f(s)=
	\frac{1}{m(B_{r})}
	\int _{B_{r}}f(st)\,dm(t),\quad f\in L_{p}
	\bigl(G;L_{p}( \mathcal{M}) \bigr),s\in G,r>0.
	\end{equation*}
	Since the condition \eqref{eq:doubling def} holds, $G$ must be unimodular
	(see, e.g., \cite{calderon53erg}). In other words, the measure $m$ is also
	invariant under left translation. Note that $sB_{r} =B(s ,r)$ by the invariance
	of~$d$. Thus $m(B_{r})=m(sB_{r})=m(B(s ,r))$ and
	\begin{equation*}
	A_{r}'f(s)=
	\frac{1}{m(B(s ,r))}
	\int _{B(s ,r)} f(t)\,dm(t).
	\end{equation*}
	And by Theorem~\ref{thm:hl max}, the right-hand side is of weak type
	$(1,1)$ and of strong type $(p,p)$ for $1<p<\infty $. Thus
	$(A_{r}')_{r>0}$ is of weak type $(1,1)$ and of strong type $(p,p)$ for
	$1<p<\infty $ as well. The theorem is proved.%
\end{proof}

\begin{example}%
	\label{ex:max doubling}
	The theorem comprises noncommutative
	variants of classical results due to Wiener \cite{weiner39ergodic}, Calder\'{o}n
	\cite{calderon53erg}, and Nevo \cite{nevo06surveyerg}. If $G$ is a compactly
	generated group of polynomial growth, then the theorem applies to a large
	class of invariant metrics on $G$, such as distance functions derived from
	an invariant Riemann metric or word metrics. We refer to
	\cite[Sections~4 and~5]{nevo06surveyerg} for more examples.\nocite{pisier16chapterbmo}
	Here, we list several typical examples satisfying the doubling condition.
	\begin{enumerate}
		\item[(1)] Let $G$ be a compactly generated group of polynomial growth,
		and let $V$ be a symmetric compact generating subset. The
		\emph{word metric} defined by
		\begin{equation*}
		\forall s,t\in G,\quad d(s,t)=\inf \{n\in \mathbb{N}, s^{-1}h
		\in V^{n} \}
		\end{equation*}
		satisfies \eqref{eq:doubling def} and \eqref{eq:asymp inv}. Note that the
		integer groups and finitely generated nilpotent groups are of polynomial
		growth.
		\begin{enumerate}
			\item[(i)] The averaging operators
			\begin{equation*}
			A_{n}x=
			\frac{1}{m(V^{n})}
			\int _{V^{n}}\alpha _{s}x\,dm(s),\quad x\in
			L_{p}( \mathcal{M}),n\in \mathbb{N}
			\end{equation*}
			are of strong type $(p,p)$ if $1<p<\infty $. If, in addition,
			$\alpha $ satisfies the condition \emph{$(\mathbf{A} ')$}, then
			$(A_{n})_{n\geq 1}$ is of weak type $(1,1)$. This in particular establishes
			the maximal inequalities in Theorem~\ref{thm:main poly}.
			\item[(ii)] Let $T:L_{p}(\mathcal{M}) \to L_{p}(\mathcal{M})$ be a positive
			invertible operator with positive inverse such that
			$\sup _{n\in \mathbb{Z}} \llVert  T^{n} \rrVert  <\infty $. Then\vspace*{-2.5pt}
			\begin{equation*}
			A_{n}=
			\frac{1}{2n+1}
			\sum _{k=-n}^{n}T^{k},\quad
			n\in \mathbb{N}
			\end{equation*}
			is of strong type $(p,p)$ if $1<p<\infty $. If $T$ is an automorphism of
			$\mathcal{M}$ which leaves $\tau $ invariant, then
			$(A_{n})_{n\geq 1}$ is of weak type $(1,1)$.
			\item[(2)] Let $G$ be a compactly generated group of polynomial growth,
			and let $d$ be a metric on~$G$. If $d$ is invariant under a cocompact subgroup
			of $G$ and if $d$ satisfies a weak kind of the existence of geodesics axiom
			(see \cite[Definition~4.1]{breuillard14polygrowth}), then $(G,d)$ satisfies
			\eqref{eq:doubling def} and \eqref{eq:asymp inv}.\vspace*{-2.5pt}
		\end{enumerate}
	\end{enumerate}
\end{example}
We remark that a natural generalization of the doubling condition
\eqref{eq:doubling def} is given by Tempelman
\cite{tempelman67doubling} as follows. A~sequence
$(F_{k})_{k\geq 1}$ of sets of finite measure in $G$ satisfies
\emph{Tempelman's regular condition} if\vspace*{-2.5pt}
\begin{equation*}
m(F_{k}^{-1}F_{k})<Cm(F_{k})
\end{equation*}
for some $C>0$ independent of~$k$. We refer to
\cite[Chapter~5]{tempelman92ergbook} for more details. It is unclear for
us how to establish the noncommutative maximal inequalities in this setting.
In the following, we provide a typical example for which the inequalities
hold true.\vspace*{-2.5pt}
\begin{thm}
	\label{thm:max subgp}%
	Let $G$ be an increasing union of compact subgroups
	$(G_{n})_{n\geq 1}$. Fix $1\leq p<\infty $. Let $(A_{n})_{n\geq 1}$ be
	the averaging operators\vspace*{-2.5pt}
	\begin{equation*}
	A_{n}x=
	\frac{1}{m(G_{n})}
	\int _{G_{n}}\alpha _{s}x\,dm(s),\quad x\in
	L_{p}( \mathcal{M}).
	\end{equation*}
	Then $(A_{n})_{n\geq 1}$ is of strong type $(p,p)$ if $1<p<\infty $. If,
	in addition, $\alpha $ satisfies the condition \emph{$(\mathbf{A} ')$}, then
	$(A_{n})_{n\geq 1}$ is of weak type $(1,1)$.\vspace*{-2.5pt}
\end{thm}
\begin{proof}
	By Theorems~\ref{thm:transference} and~\ref{thm:transference weak}, it
	suffices to prove the maximal inequalities for the averaging operators\vspace*{-2.5pt}
	\begin{equation*}
	A_{n}'f(s)=
	\frac{1}{m(G_{n})}
	\int _{G_{n}}f(st)\,dm(t),\quad f\in L_{p}
	\bigl(G;L_{p}( \mathcal{M}) \bigr),s\in G,n\geq 1.
	\end{equation*}
	Set $\Sigma $ to be the $\sigma $-algebra of Borel sets on~$G$. For each
	$n\geq 1$, we define $\Sigma _{n}\subset \Sigma $ to be the
	$\sigma $-subalgebra generated by the cosets of $G_{n}$\vspace*{-2.5pt}
	\begin{equation*}
	\{s G_{n}:s\in G\}.
	\end{equation*}
	We see that $\Sigma _{n+1}\subset \Sigma _{n}$ for all $n\geq 1$. Let
	$\mathbb{E}_{n}$ be the conditional expectation from
	$L_{\infty }(G,\Sigma ,m)\bar{\otimes }\mathcal{M}$ to
	$L_{\infty }(G,\Sigma _{n},m| _{\Sigma _{n}})\bar{\otimes }
	\mathcal{M}$. Then it is easy to see that
	\begin{equation*}
	\mathbb{E}_{n}=A_{n}',\quad n\geq 1.
	\end{equation*}
	According to Lemma~\ref{lem:doob}, we see that
	$(\mathbb{E}_{n})_{n\geq 1}$ is of weak type $(1,1)$ and of strong type
	$(p,p)$ for $1<p<\infty $. This yields the desired inequalities.
\end{proof}

\subsection{A~random walk approach}
\label{sect:walk}

In this subsection, we provide an alternative approach to maximal inequalities
for groups of polynomial growth. This approach is based on a Gaussian lower
bound of random walks on groups (see
\cite{hebischsaloffcoste93gaussianbound}). Independent of the previous
approaches, in this method we do not need the results on dyadic decompositions
of the group, nor do we use the transference principle. The key observation
is that we may relate the ball averages on groups with the ergodic averages
of a Markov operator.

\begin{prop}%
	\label{prop:walk domination}
	Let $G$ be a locally compact group of polynomial growth, and let $V$ be
	a compact generating set. Let $\alpha $ be a strongly continuous action
	of $G$ on an ordered Banach space $E$ such that
	$\alpha _{s} x\geq 0$ for all $s\in G$ and $x\in E_{+}$. Define an operator
	$T$ on $E$ by
	\begin{equation*}
	Tx=
	\frac{1}{m(V)}
	\int _{V} \alpha _{s} x \,dm(s),\quad x\in E.
	\end{equation*}
	Then there exists a constant $c$ depending only on $G$ such that
	\begin{equation*}
	\frac{1}{m(V^{n})}
	\int _{V^{n}} \alpha _{s} x \,dm(s) \leq
	\frac{c}{n^{2}}
	\sum _{k=1}^{2n^{2}}T^{k} x,
	\quad x\in E_{+}.
	\end{equation*}
\end{prop}

We remark that for actions by abelian semigroups, it is known by Brunel
\cite{brunel73ergzn} that the ergodic averages of multioperators can be
related to averages of some Markov operators; Nevo and Stein
\cite{nevostein94freeerg} showed that similar observations hold for spherical
averages of free group actions. These results play an essential role in
the proof of ergodic theorems therein. In the case of groups of polynomial
growth, our construction of Markov operators is different from theirs,
which is inspired by \cite{steinstromberg83maxlargen}. The argument is
relatively easy and based on the Markov chains on groups of polynomial
growth.

To prove the proposition, we consider a locally compact group $G$ and a
measure $\mu $ on~$G$. For an integer $k$, we denote by
$\mu ^{\star k}$ the $k$th convolution of $\mu $, that is, the unique measure
$\nu $ on $G$ satisfying
\begin{equation*}
\int _{G} f \,d\nu =
\int _{\prod _{i=1}^{k} G}f(s_{1}\cdots s_{k}) \,d \mu
(s_{1})\cdots d\mu (s_{k}),\quad f\in C_{0}(G).
\end{equation*}
If $f$ is the density function of $\mu $, then we still denote by
$f^{\star k}$ the density function of $\mu ^{\star k}$.

In the following $d$ will denote the word metric with respect to $V$ introduced
in Example~\ref{ex:max doubling}(1), and
$B_{r}= \{x\in G:d(e,x)\leq r\}$ for $r>0$.
\begin{lem}[{\cite[Theorem 5.1]{hebischsaloffcoste93gaussianbound}}]
	Let $G$ be a
	locally compact group of polynomial growth, and let $V$ be a compact generating
	set. Let $f$ be the density function of a symmetric continuous probability
	measure on $G$ such that $\operatorname{supp} (f)$ is bounded and
	$V\subset \operatorname{supp} (f)$. Then there exists a constant $c>0$ such that
	for any integer~$k$,
	\begin{equation*}
	f^{\star k}(s)\geq
	\frac{c e^{-d(e,s)^{2}/k}}{m(B_{\sqrt{k}})},\quad s \in B_{k}.
	\end{equation*}
\end{lem}

\begin{lem}%
	\label{lem:walk}
	Let $G$, $V$, and $f$ be as in the previous lemma. Then there exists a
	constant $c>0$ such that for any integer~$n$,
	\begin{equation*}
	\frac{\chi _{B_{n }}}{m(B_{n})}\leq
	\frac{c}{2n^{2}}
	\sum _{k=1}^{2n^{2}}f^{
		\star k}.
	\end{equation*}
\end{lem}

\begin{proof}
	It suffices to prove the inequality in the lemma for sufficiently large~$n$. By the previous lemma, there exists $c>0$ such that
	\begin{equation*}
	f^{\star k}(s)\geq
	\frac{c}{m(B_{\sqrt{k}})},\quad s\in B_{\sqrt{k}}(e).
	\end{equation*}
	Therefore, for $s\in B_{n }(e)$,
	\begin{align*}
	\frac{1}{2n^{2}}
	\sum _{k=1}^{2n^{2}}f^{\star k}(s) &
	\geq
	\frac{1}{2n^{2}}
	\sum _{k=(n+1)^{2}}^{2n^{2}}f^{\star k}(s)
	\geq
	\frac{1}{2n^{2}}
	\sum _{k=(n+1)^{2}}^{2n^{2}}
	\frac{c}{m(B_{\sqrt{k}})}
	\\
	&\geq
	\frac{c(n^{2}-2n-1)}{2n^{2}m(B_{\sqrt{2}n})} \geq
	\frac{c'}{m(B_{n})},
	\end{align*}
	where $c'>0$ is a constant depending only on the doubling condition of~$G$.
\end{proof}

\begin{proof}[Proof of Proposition~\ref{prop:walk domination}]
	We apply Lemma~\ref{lem:walk} with $f=\chi _{V}/m(V)$. Then we obtain for
	$x\in E_{+}$,
	\begin{equation*}
	\frac{1}{m(V^{n})}
	\int _{V^{n}} \alpha _{s} x \,dm(s) \leq
	\frac{c}{2n^{2}}
	\sum _{k=1}^{2n^{2}}
	\int _{G} \alpha _{s} x f^{\star k}(s)
	\,dm(s).
	\end{equation*}
	By the definition of $f^{\star k}$, we have
	\begin{equation*}
	\int _{G} \alpha _{s} x f^{\star k}(s)
	\,dm(s) =
	\int _{\prod _{i=1}^{k} G} \alpha _{s_{1}\cdots s_{k}}x f (s_{1})\cdots
	f(s_{k}) \,dm(s_{1}) \cdots dm(s_{k}).
	\end{equation*}
	Recall that $\alpha $ is a group action and $f=\chi _{V}/m(V)$, so we obtain
	\begin{align*}
	\int _{G} \alpha _{s} x f^{\star k}(s)
	\,dm(s) =
	\frac{1}{m(V)^{k}}
	\int _{
		\prod _{i=1}^{k} V} \alpha _{s_{1}}\cdots \alpha
	_{s_{k}}x \,dm(s_{1}) \cdots dm(s_{k})=T^{k}
	x.
	\end{align*}
	We have therefore established the desired inequality.
\end{proof}

Proposition~\ref{prop:walk domination} allows us to deduce maximal ergodic
theorems for group actions from well-known ergodic theorems for Markov
operators. As a corollary, we may obtain the maximal inequalities in Example~\ref{ex:max doubling}(1). Moreover, for maximal ergodic inequalities on
$L_{1}$, it is not necessary to assume that $\alpha $ is an action by automorphisms
as in $(\mathbf{A} ')$.
\begin{cor}%
	\label{cor:weakonepoly}
	Let $G$ and $V$ be as above. Let $\alpha $ be a continuous $\tau $-preserving
	action of $G$ on $\mathcal{M}$ such that $\alpha _{s}$ is a positive isometry
	on $\mathcal{M}$ for each $s\in G$. Then $\alpha $ extends to an action
	on $L_{1}(\mathcal{M})$. The operators defined by
	\begin{equation*}
	A_{n}x=
	\frac{1}{m(V^{n})}
	\int _{V^{n}}\alpha _{s}x\,dm(s),\quad x\in
	L_{1}( \mathcal{M}),n\in \mathbb{N}
	\end{equation*}
	are of weak type $(1,1)$.
\end{cor}
\begin{proof}
	Note that the operator
	\begin{equation*}
	Tx=
	\frac{1}{m(V)}
	\int _{V} \alpha _{s} x \,dm(s),\quad x\in
	\mathcal{M}
	\end{equation*}
	is a positive contraction on $\mathcal{M}$, which preserves~$\tau$. Then
	it is well known that the averages
	$\frac{1}{n }\sum _{k=1}^{n}T^{k} $ are of weak type $(1,1)$ (see
	\cite{yeadon77max}). Thus, by Proposition~\ref{prop:walk domination},
	$(A_{n})_{n\geq 0}$ is of weak type $(1,1)$ as well.
\end{proof}

\section{Maximal inequalities: Group-theoretic approach}
\label{sect:group}

In this section we provide an alternative approach to Theorem~\ref{thm:max doubling} in the case where $G$ is a finitely generated discrete
group of polynomial growth, $d$ is the word metric, and $p\neq 1$. The
argument follows from a structural study of nilpotent groups.\vadjust{\goodbreak}

We first recall some well-known facts on the structure of nilpotent groups.
Let $G$ be a discrete finitely generated nilpotent group with lower central
series
\begin{equation*}
G=G_{1}\supset G_{2}\supset \cdots \supset G_{K}
\supset G_{K+1}=\{e\}.
\end{equation*}
Each quotient group $G_{i}/G_{i+1}$ is an abelian group of rank
$r_{i}$, that is, there is a group isomorphism
\begin{equation*}
\pi _{i}:G_{i}/G_{i+1}\to F_{i}\times
\mathbb{Z}^{r_{i}}
\end{equation*}
with a finite abelian group $F_{i}$. It was shown in
\cite{bass72polygrowthnil} that $G$ is of polynomial growth. We summarize
below some facts in the argument of \cite{bass72polygrowthnil}. We may
choose a finite generating set $T$ of $G$ such that
\begin{equation*}
[T,T]=\{s^{-1}t^{-1}st:s,t\in T\}\subset T
\end{equation*}
and take
\begin{equation*}
T_{i}=G_{i}\cap T,\quad 1\leq i\leq K+1.
\end{equation*}
Then
\begin{equation*}
G_{i}=\langle T_{i}\rangle ,\quad 1\leq i\leq K+1.
\end{equation*}
For each $1\leq j\leq K$, we order the elements in
$T_{j}\setminus T_{j+1}$ as
\begin{equation}
\label{eq:tjminus} T_{j}\setminus T_{j+1} =
\{t_{1}^{(j)},t_{2}^{(j)},
\ldots,t_{r_{j}}^{(j)},t_{r_{j}+1}^{(j)}
,\ldots,t_{l_{j}}^{(j)}\}
\end{equation}
so that
$\pi _{j}([t_{1}^{(j)}]),\ldots,\pi _{j}([t_{r_{j}}^{(j)}])$ are the generators
of $\mathbb{Z}^{r_{j}}$. Let $N_{j}$ be the index of the subgroup
$\langle t_{1}^{(j)},t_{2}^{(j)},\ldots,t_{r_{j}}^{(j)}\rangle G_{j+1}$
in $G_{j}$.

By a \emph{word} in a subset $T'\subset T$ we mean a sequence of elements
$w=(s_{1},s_{2},\ldots,\allowbreak s_{n})$ with
$s_{1},s_{2},\ldots,s_{n}\in T'$, and we denote by
$ \llvert  w \rrvert  =s_{1}s_{2}\cdots s_{n}\in G$ the resulting group element in~$G$. If
$w=(s_{1},s_{2},\ldots,s_{n})$ is a word in $T'$ and if
$T''\subset T'$, then we let $\deg _{T''}(w)$ be the cardinality of
$\{k:1\leq k\leq n,s_{k}\in T'' \}$. We say that
\begin{equation*}
\deg ^{(j)}(w)\leq (d_{j},d_{j+1},
\ldots,d_{K})\coloneqq d
\end{equation*}
if we have $\deg _{T_{j}\setminus T_{i+1}}(w)\leq d_{i}$ for all
$j\leq i\leq K$. Denote by $G_{j}(d)$ the set of all words $w$ in
$T_{j}$ such that $\deg ^{(j)}(w)\leq d$, and denote by $G_{j}'(d)$ the
subset of the words $w\in G_{j}(d)$ of the form
\begin{equation*}
w=(t_{1}^{(j)},\ldots,t_{1}^{(j)},
\ldots,t_{l_{j}}^{(j)},\ldots,t_{l_{j}}^{(j)},v),
\end{equation*}
where $v$ is a word in $T_{j+1}$ and the element $t_{k}^{(j)}$ does not
appear more than $N_{j}$ times in $w$ for $r_{j}<k\leq l_{j}$.

The key observation in \cite{bass72polygrowthnil} for proving the polynomial
growth of $G$ is as follows (see the assertions (6) and (7) in
\cite[p.~613]{bass72polygrowthnil}).
\begin{lem}
	Let $c>0$ be a constant, and let $m\geq 1$. For each $1\leq j\leq K$, we
	have
	\begin{equation*}
	\begin{aligned}
	&
	\bigl\{g\in G:g= \llvert w \rrvert ,w\in G_{j}(cm^{j},
	\ldots,cm^{K}) \bigr\}
	\\
	&\quad\quad  \subset \bigl\{g\in G:g= \llvert w
	\rrvert ,w \in G_{j}'(c'm^{j},\ldots,c'm^{K})
	\bigr\},
	\end{aligned}
	\end{equation*}
	where $c'>0$ is a constant depending only on $c$ and $G_{j}$.
\end{lem}
In particular, we take $g\in T^{n}$. Hence $g$ corresponds to a word
$w$ in $T=T_{1}$ such that
\begin{equation*}
g= \llvert w \rrvert , \quad w\in G_{1}(m_{1},
\ldots,m_{K}), m\coloneqq \max \{m_{1},\ldots,m_{K}
\}\leq n.
\end{equation*}
Using the lemma inductively, we may find another word
$w'=(w_{1},\ldots,w_{K})$ in $T$ (where each $w_{j}$ in the bracket stands
for a subword in $T_{j}\setminus T_{j+1}$) and a constant $c>0$ such that
\begin{equation*}
g= \llvert w' \rrvert , \quad (w_{j},\ldots,w_{K})\in
G_{j}'(cn,\ldots,cn^{K}), 1 \leq j\leq K.
\end{equation*}
In other words, we have the following observation.
\begin{lem}
	\label{lem:sn nilpotent}%
	Let $n\in \mathbb{N}$. Each element $g\in T^{n}$ can be written in the
	form
	\begin{equation*}
	g=(t_{1}^{(1)})^{n_{11}}\cdots
	(t_{l_{1}}^{(1)})^{n_{1l_{1}}}\cdots
	(t_{1}^{(K)})^{n_{K1}} \cdots
	(t_{l_{K}}^{(K)})^{n_{Kl_{K}}},
	\end{equation*}
	where there exists a constant $c>0$ such that for $1\leq j\leq K$,
	\begin{equation*}
	n_{jk}\leq cn^{j}, \quad \text{if }1\leq k\leq
	r_{j},
	\end{equation*}
	and
	\begin{equation*}
	n_{jk}\leq N_{j}, \quad \text{if }r_{j}<k\leq
	l_{j}.
	\end{equation*}
\end{lem}
In \cite{bass72polygrowthnil} and \cite{wolf68sovablegrowth} it is proved
that $G$ satisfies the following strict polynomial growth condition.
\begin{lem}
	\label{lem:strcit growth nil}%
	We have two constants $c_{1},c_{2}>0$ such that
	\begin{equation*}
	c_{1}n^{d(G)}\leq \llvert T^{n}
	\rrvert \leq c_{2}n^{d(G)},
	\end{equation*}
	where $ \llvert  \, \rrvert  $ denotes the cardinality of a subset and\vspace*{-3pt}
	\begin{equation*}
	d(G)=
	\sum _{j=1}^{K}jr_{j}.
	\end{equation*}
\end{lem}
Note that the upper bound in the above lemma follows directly from Lemma~\ref{lem:sn nilpotent}.

Now we will prove the following maximal inequalities, which are particular
cases studied in Theorem~\ref{thm:max doubling} and Example~\ref{ex:max doubling}(1).\vspace*{-3pt}
\begin{prop}
	\label{thm:p-p max}%
	Let $G$ be a finitely generated discrete group of polynomial growth, and
	let $S\subset G$ be a finite generating set. Fix $1<p<\infty $, and let
	$\alpha $ be an action $\alpha =(\alpha _{s})_{s\in G}$ of $G$ on
	$L_{p}(\mathcal{M})$ which satisfies
	\emph{$(\mathbf{A} ^{p}_{1})$--$(\mathbf{A} ^{p}_{3})$}. We consider the averaging
	operators\vspace*{-3pt}
	\begin{equation*}
	A_{n}=
	\frac{1}{ \llvert  S^{n} \rrvert  }
	\sum _{s\in S^{n}}\alpha _{s},\quad n\geq 1,
	\end{equation*}
	where $ \llvert  \, \rrvert  $ denotes the cardinality of a subset. Then
	$(A_{n})_{n\geq 1}$ is of strong type $(p,p)$.\vspace*{-3pt}
\end{prop}
The proposition relies on the following characterization of groups of polynomial
growth by Gromov \cite{gromov81polygrowth}.\vspace*{-3pt}
\begin{lem}
	\label{lem:gromov thm}%
	Any finitely generated discrete group of polynomial growth contains a finitely
	generated nilpotent subgroup of finite index.\vspace*{-3pt}
\end{lem}
We also need the following fact.\vspace*{-3pt}
\begin{lem}
	\label{lem:from nilpotent to general}%
	Let $G$ be a finitely generated group of polynomial growth. Let $H$ be
	a normal subgroup of $G$ of finite index. Then $H$ is finitely generated.
	Let $U\subset G$ be a finite system of representatives of the cosets
	$G/H$ with $e\in U$. Let $T\subset H$ be a finite generating set of~$H$. Write $V=U\cup T$. Then there exists an integer $N$ such
	that\vspace*{-3pt}
	\begin{equation*}
	\forall m\in \mathbb{N},\quad V^{m}\subset
	UT^{(3N+1)m}.
	\end{equation*}
\end{lem}
\begin{proof}
	This is given in the proof of
	\cite[Theorem~3.11]{wolf68sovablegrowth}. Let $N$ be an integer large enough
	such that for all $u_{1}^{\epsilon },u_{2}^{\eta }\in U$ with
	$\epsilon ,\eta \in \{\pm 1\}$, there exist $u\in U$ and
	$t\in T^{N}$ satisfying $u_{1}^{\epsilon }u_{2}^{\eta }=ut$. Then $N$ satisfies
	the desired condition.\vspace*{-3pt}
\end{proof}
Now we deduce the desired result.\vadjust{\goodbreak}
\begin{proof}[Proof of Proposition~\ref{thm:p-p max}]
	By Theorem~\ref{thm:transference}, it suffices to consider the case where
	$\alpha $ is an action on $L_{p}(G;\break L_{p}(\mathcal{M}))$ by translation.
	By Lemma~\ref{lem:gromov thm}, we may find a nilpotent subgroup
	$H\subset G$ of finite index. As is explained in
	\cite[Theorem~3.11]{wolf68sovablegrowth}, $H$ can be taken normal by replacing
	$H$ with $\bigcap _{s\in G}sHs^{-1}$. Now let
	$T=\{t_{k}^{(j)}:1\leq k\leq l_{j},1\leq j\leq K\}$ be a finite generating
	set of the nilpotent group $H$, where $T$ and the indices $k$, $j$ are chosen
	in the same manner as in \eqref{eq:tjminus}. Also, let $U$ and $V$ be given
	as in the previous lemma. Consider
	$x\in L_{p}^{+}(G;L_{p}(\mathcal{M}))$, and write
	\begin{equation*}
	\tilde{A}_{n}x=
	\frac{1}{ \llvert  V^{n} \rrvert  }
	\sum _{s\in V^{n}}\alpha _{s}x,\quad n \in \mathbb{N}.
	\end{equation*}
	Since the operators $\alpha _{s}$ extend to positive operators on
	$L_{p}(G;L_{p}(\mathcal{M}))$, by Lemma~\ref{lem:sn nilpotent} and Lemma~\ref{lem:from nilpotent to general}, there exists a constant $c>0$ for
	all $n\in \mathbb{N}$ such that
	\begin{align*}
	\sum _{s\in V^{n}}\alpha _{s}x &\leq
	\sum _{h\in U}\alpha _{h}
	\sum _{t
		\in T^{(3N+1)n}}\alpha _{t}x
	\\
	&\leq
	\sum _{h\in U}\alpha _{h}
	\sum _{1\leq j\leq K}
	\sum _{
		\substack{
			1\leq n_{jk}\leq cn^{j}
			\\
			1\leq k\leq r_{j}
		}
	}
	\sum _{
		\substack{
			1\leq n_{jk}\leq N_{j}
			\\
			r_{j}<k\leq l_{j}
		}
	}\alpha _{t_{1}^{(1)}}^{n_{11}}\cdots
	\alpha _{t_{l_{1}}^{(1)}}^{n_{1l_{1}}} \cdots \alpha
	_{t_{1}^{(K)}}^{n_{K1}}\cdots \alpha _{t_{l_{K}}^{(K)}}^{n_{Kl_{K}}}x.
	\end{align*}
	Recall that by Lemma~\ref{lem:strcit growth nil} we may find a constant
	$c'>0$ such that
	\begin{equation*}
	\llvert V^{n} \rrvert \geq c'n^{\sum _{j=1}^{K}jr_{j}}.
	\end{equation*}
	So we may find a constant $c''>0$ satisfying
	\begin{equation*}
	\begin{aligned}
	\tilde{A}_{n}x\leq{}& c''
	\sum _{h\in U}\alpha _{h}
	\frac{1}{n^{\sum _{j=1}^{K}jr_{j}}}
	\\
	&{}\times
	\sum _{1\leq j\leq K}
	\sum _{
		\substack{
			1\leq n_{jk}\leq cn^{j}
			\\
			1\leq k\leq r_{j}
		}
	}
	\sum _{
		\substack{
			1\leq n_{jk}\leq N_{j}
			\\
			r_{j}<k\leq l_{j}
		}
	}\alpha _{t_{1}^{(1)}}^{n_{11}}\cdots
	\alpha _{t_{l_{1}}^{(1)}}^{n_{1l_{1}}} \cdots \alpha
	_{t_{1}^{(K)}}^{n_{K1}}\cdots \alpha _{t_{l_{K}}^{(K)}}^{n_{Kl_{K}}}x.
	\end{aligned}
	\end{equation*}
	Note that by \cite[Theorem~4.1]{jungexu07erg}, for each
	$1\leq j\leq K$ and $1\leq k\leq r_{j}$ there exists a constant
	$C_{p}$ depending only on $p$ such that
	\begin{equation*}
	\Bigl\llVert {
		\sup _{n}}^{+}
	\frac{1}{cn^{j}}
	\sum _{l=1}^{cn^{j}}\alpha _{t_{k}^{(j)}}^{l}x
	\Bigr\rrVert _{p}\leq C_{p} \llVert x \rrVert _{p},
	\quad x\in L_{p} \bigl(G;L_{p}(\mathcal{M}) \bigr).
	\end{equation*}
	Applying the inequality iteratively, we obtain a constant $C_{p}'>0$ such
	that
	\begin{equation*}
	\bigl\llVert {
		\sup _{n}}^{+}\tilde{A}_{n}x \bigr
	\rrVert _{p}\leq C_{p}' \llVert x \rrVert _{p},
	\quad x\in L_{p} \bigl(G;L_{p}(\mathcal{M}) \bigr).
	\end{equation*}
	Since $S$ and $V$ are both finite, we may find two integers $k$ and
	$k'$ with
	\begin{equation*}
	S\subset V^{k},\quad\quad V\subset S^{k'}.
	\end{equation*}
	So the strong type $(p,p)$ inequality for $A_{n}$ follows as well.
\end{proof}

\section{Individual ergodic theorems}
\label{sect:individual}

In this section, we apply the maximal inequalities to study the pointwise
ergodic convergence in Theorems~\ref{thm:main intro} and~\ref{thm:main poly}.

We will use the following analogue for the noncommutative setting of the
usual almost everywhere convergence. The definition is introduced by Lance
\cite{lance76erg} (see also \cite{jajte85nclimitbook}).
\begin{defn}%
	\label{defn:au convergence}
	Let $\mathcal{M}$ be a von Neumann algebra equipped with a normal
	semifinite faithful trace~$\tau$. Let
	$x_{n},x\in L_{0}(\mathcal{M})$. We say that
	$(x_{n})_{n\geq 1}$ converges
	\emph{bilaterally almost uniformly} (\emph{b.a.u.} for short) to $x$ if for
	every $\varepsilon >0$ there is a projection
	$e\in \mathcal{M}$ such that
	\begin{equation*}
	\tau (e^{\perp })<\varepsilon \quad\quad \text{and}\quad\quad
	\lim _{n\to \infty } \bigl\llVert e(x_{n}-x)e \bigr\rrVert
	_{\infty }=0,
	\end{equation*}
	and that it converges \emph{almost uniformly} (\emph{a.u.} for short) to
	$x$ if for every $\varepsilon >0$ there is a projection
	$e\in \mathcal{M}$ such that
	\begin{equation*}
	\tau (e^{\perp })<\varepsilon \quad\quad \text{and}\quad\quad
	\lim _{n\to \infty } \bigl\llVert (x_{n}-x)e \bigr\rrVert
	_{\infty }=0.
	\end{equation*}
\end{defn}

In the case of classical probability spaces, the definition above is equivalent
to the usual almost everywhere convergence in terms of Egorov's theorem.

Now let $G$ be an amenable locally compact group, and let
$(F_{n})_{n\geq 1}$ be a F{\o}lner sequence in~$G$. Let
$1\leq p\leq \infty $. Assume that $\alpha =(\alpha _{s})_{s\in G}$ is
an action on $L_{p}(\mathcal{M})$ which satisfies
$(\mathbf{A} ^{p}_{1})$--$(\mathbf{A} ^{p}_{3})$. Denote by $A_{n}$ the corresponding
averaging operators
\begin{equation*}
A_{n}x=
\frac{1}{m(F_{n})}
\int _{F_{n}}\alpha _{s}x\,dm(s),\quad x\in
L_{p}( \mathcal{M}).
\end{equation*}
We keep the notation $\mathcal{F}_{p}\subset L_{p}(\mathcal{M})$ and
$P$ introduced in Section~\ref{sub:Actions-by-amenable}.

We first consider the case where
$\alpha $ extends to an action on $L_{1}(\mathcal{M})+\mathcal{M}$. In this
case we use a standard argument for b.a.u. convergences adapted from
\cite{hong17dimfree,jungexu07erg}, and \cite{yeadon77max}. The
following lemma from \cite{defantjunge04maxau} will be useful.

\begin{lem}%
	\label{lem:czero}
	Let $1\leq p<\infty $. If $(x_{n})\in L_{p}(\mathcal{M}; c_{0})$, then
	$x_{n}$ converges b.a.u. to~$0$. If
	$(x_{n})\in L_{p}(\mathcal{M}; c_{0}^{c})$ with $2\leq p<\infty $, then
	$x_{n}$ converges a.u. to~$0$.
\end{lem}

We will also use the following noncommutative analogue of the Banach principle
given by \cite{litvinov17au} and
\cite[Theorem~3.1]{chilinlitvinov16ausymmetric}.

\begin{lem}%
	\label{lem:principle litvinov}
	Let $1\leq p<\infty $, and let $S=(S_{n})_{n\geq 1}$ be a sequence of additive
	maps from $L_{p}^{+}(\mathcal{M})$ to
	$L_{0}^{+}(\mathcal{M})$. Assume that $S$ is of weak type
	$(p,p)$. Then the set
	\begin{equation*}
	\mathcal{C} = \bigl\{ x\in L_{p}(\mathcal{M}) : ( S_{n} x
	)_{n
		\geq 1} \text{ converges a.u.} \bigr\}
	\end{equation*}
	is closed in $L_{p}(\mathcal{M})$.
\end{lem}

\begin{prop}
	\label{prop:individual}
	Assume that $\alpha =(\alpha _{s})_{s\in G}$ is an action well defined
	on $\bigcup _{1\leq p\leq \infty }L_{p}(\mathcal{M})$ which satisfies
	\emph{$(\mathbf{A} ^{p}_{1})$--$(\mathbf{A} ^{p}_{3})$} for every
	$1\leq p\leq \infty $. Let $(A_{n})_{n\geq 1}$ be as above, and let
	$1\leq p_{0}<p_{1}\leq \infty $. Assume that $(A_{n})_{n\geq 1}$ is of
	strong type $(p,p)$ for all $p_{0}<p<p_{1}$.
	\begin{enumerate}
		\item[(1)] For all $x\in L_{p}(\mathcal{M})$ with
		$p_{0}<p\leq 2p_{0}$,
		$(A_{n}x-Px)_{n\geq 1}\in L_{p}(\mathcal{M};c_{0})$, and hence
		$(A_{n} x)_{n\geq 1}$ converges b.a.u. to $Px$.
		\item[(2)] For all $x\in L_{p}(\mathcal{M})$ with
		$2p_{0}< p < p_{1}$,
		$(A_{n}x-Px)_{n\geq 1}\in L_{p}(\mathcal{M};c_{0}^{c})$, and
		hence $(A_{n} x)_{n\geq 1}$ converges a.u. to $Px$.
	\end{enumerate}
\end{prop}
\begin{proof}
	According to the splitting \eqref{eq:decomposition erg} and the discussion
	after it, we know that
	\begin{equation*}
	S=\operatorname{span} \bigl\{x-\alpha _{s}x:s\in G,x\in L_{1}(
	\mathcal{M})\cap \mathcal{M} \bigr\}
	\end{equation*}
	is dense in
	$\overline{(\mathrm{Id}-P)(L_{p}(\mathcal{M}))}$ for all
	$1\leq p<\infty $. Also, observe that for all $x\in S$,
	\begin{equation}
	\lim _{n\to \infty }A_{n}x=0 \text{ a.u.},\quad x\in S.
	\label{eq:limit an on s}
	\end{equation}
	To see this, take an arbitrary $x\in S$ of the form
	$x=y-\alpha _{s_{0}}y$ for some $s_{0}\in G$ and
	$y\in L_{1}(\mathcal{M})\cap \mathcal{M}$. Then
	\begin{align*}
	A_{n}x&=
	\frac{1}{m(F_{n})}
	\int _{F_{n}}(\alpha _{s}y-\alpha _{ss_{0}}y)
	\,dm(s)
	\\
	&=
	\frac{1}{m(F_{n})}
	\int _{F_{n}\setminus ( F_{n}s_{0}\cap F_{n})} \alpha _{s}y\,dm(s) -
	\frac{1}{m(F_{n})}
	\int _{F_{n}s_{0}\setminus ( F_{n}s_{0}
		\cap F_{n})}\alpha _{s}y\,dm(s) .
	\end{align*}
	Therefore, according to $(\mathbf{A} _{2}^{\infty })$,
	\begin{equation}
	\label{eq:estimatebyfolner} \llVert A_{n}x \rrVert _{\infty }\leq
	\frac{m(F_{n}\bigtriangleup F_{n}s_{0})}{m(F_{n})} \llVert y \rrVert _{\infty }
	\sup _{s
		\in G} \llVert \alpha _{s} \rrVert
	_{B(\mathcal{M})},
	\end{equation}
	which converges to $0$ as $n\to \infty $ according to the F{\o}lner condition.
	This therefore yields the a.u. convergence of $(A_{n} x)_{n}$ in
	\eqref{eq:limit an on s}, as desired.
	
	Now we prove assertion (1). Take $x\in L_{p}(\mathcal{M})$. Since
	$S$ is dense in\break
	$\overline{(\mathrm{Id}-P)(L_{p}(\mathcal{M}))}$, there are
	$x_{k}\in S$ such that
	\begin{equation*}
	\lim _{k\to \infty } \llVert x-Px-x_{k} \rrVert
	_{p}=0.
	\end{equation*}
	Since $(A_{n})_{n\geq 1}$ is of strong type $(p,p)$, there exists a constant
	$C>0$ independent of $x$ such that
	\begin{equation*}
	\bigl\llVert (A_{n}x-Px-A_{n}x_{k} )_{n}
	\bigr\rrVert _{L_{p}(
		\mathcal{M};\ell _{\infty })}\le C \llVert x-Px-x_{k} \rrVert
	_{p}.
	\end{equation*}
	Thus,
	\begin{equation*}
	\lim _{k\to \infty } (A_{n}x_{k} )_{n}=
	(A_{n}x-Px )_{n} \quad \text{in } L_{p}(
	\mathcal{M};\ell _{\infty }).
	\end{equation*}
	Since $L_{p}(\mathcal{M};c_{0})$ is closed in
	$L_{p}(\mathcal{M};\ell _{\infty })$, it suffices to show that
	$  (A_{n}x_{k}  )_{n}\in L_{p}(\mathcal{M};c_{0})$ for all~$k$. To this end, we take an arbitrary $z\in S$ of the form
	$z=y-\alpha _{s_{0}}y$ for some $s_{0}\in G$ and
	$y\in L_{1}(\mathcal{M})\cap \mathcal{M}$. Take some
	$p_{0}<q<p$. Note that $z\in L_{q}(\mathcal{M})$ and that
	$(A_{n})_{n}$ is of strong type $(q,q)$ by assumption, so
	$  (A_{n}(z)  )_{n}$ belongs to
	$L_{q}(\mathcal{M};\ell _{\infty })$. Then by
	\eqref{eq:interp-vectorLp} and \eqref{eq:estimatebyfolner}, for any
	$m<n$,
	\begin{align*}
	\llVert \mathop{{\sup }^{+}}_{m\le j\le n}A_{j}z
	\rrVert _{p} &\le
	\sup _{m
		\le j\le n} \llVert A_{j}z \rrVert
	_{\infty }^{1-\frac{q}{p}} \llVert \mathop{{ \sup
		}^{+}}_{m\le j\le n}A_{j}z \rrVert
	_{q}^{\frac{q}{p}}
	\\
	&\le
	\sup _{m\le j\le n} \Bigl(
	\frac{m(F_{j}\bigtriangleup F_{j}s_{0})}{m(F_{j})} \llVert y \rrVert _{\infty } \Bigr)^{1-\frac{q}{p}}
	\llVert \mathop{{\sup }^{+}}_{m\le j\le n}A_{j}z
	\rrVert _{q}^{\frac{q}{p}}.
	\end{align*}
	Thus, $   \llVert  \mathop{{\sup }^{+}}_{j\geq m}A_{j}z   \rrVert  _{p}$ tends to
	$0$ as $m\to \infty $. Therefore, the finite sequence
	$(A_{1}z,\ldots,A_{m}z,0,\ldots )$ converges to $(A_{n}z)_{n}$ in
	$L_{p}(\mathcal{M};\ell _{\infty })$ as $m\to \infty $. As a result
	$  (A_{n}(z)  )_{n}\in L_{p}(\mathcal{M};c_{0})$, as desired.
	
	Assertion (2) is similar. It suffices to note that by the classical Kadison
	inequality in \cite{kadison52kadison},
	\begin{equation*}
	(A_{n} x)^{2}\leq A_{n}
	(x^{2})
	\sup _{s\in G} \llVert \alpha _{s} \rrVert
	_{B(
		\mathcal{M})},\quad x\in L_{p}(\mathcal{M})\cap \mathcal{M},x \geq
	0,
	\end{equation*}
	and hence by the strong type $(p,p)$ inequality and the definition of
	$L_{p}(\mathcal{M};\ell _{\infty }^{c})$, there exists a constant
	$C$ such that
	\begin{equation*}
	\bigl\llVert (A_{n} x)_{n\geq 1} \bigr\rrVert _{L_{p}(\mathcal{M};\ell _{\infty }^{c})}
	\leq \bigl\llVert \bigl((A_{n} x)^{2}
	\bigr)_{n\geq 1} \bigr\rrVert ^{1/2}_{L_{p/2}(
		\mathcal{M};\ell _{\infty })} \leq C
	\llVert x \rrVert _{p},\quad x\in L_{p}^{+}(
	\mathcal{M}).
	\end{equation*}
	Then a similar argument yields that
	$(A_{n}x-Px)_{n\geq 1}\in L_{p}(\mathcal{M};c_{0}^{c})$.
\end{proof}
\begin{rem}
	The above argument certainly works as well for a F{\o}lner sequence
	$(F_{r})_{r>0}$ indexed by $r\in \mathbb{R}_{+}$, provided that
	$r\mapsto A_{r}x$ is continuous for $x\in L_{p}(\mathcal{M})$.
\end{rem}

As a corollary, we obtain the individual ergodic theorems for actions on
$L_{1}(\mathcal{M})+\mathcal{M}$. We complete the proof of Theorem~\ref{thm:main intro}.

\begin{thm}
	Let $d$ be an invariant metric on~$G$. Assume that $(G,d)$ satisfies
	\eqref{eq:doubling def} and \eqref{eq:asymp inv}. Let $\alpha $ be an action
	of $G$ well defined on
	$\bigcup _{1\leq p\leq \infty }L_{p}(\mathcal{M})$ which satisfies
	\emph{$(\mathbf{A} ^{p}_{1})$--$(\mathbf{A} ^{p}_{3})$} for every
	$1\leq p\leq \infty $. Denote
	\begin{equation*}
	A_{r}x=
	\frac{1}{m(B_{r})}
	\int _{B_{r}}\alpha _{s}x\,dm(s),\quad x\in
	L_{p}( \mathcal{M}), r>0.
	\end{equation*}
	Then $(A_{r}x)_{r>0}$ converges a.u. to $Px$ as $r\to \infty $ for all
	$1<p<\infty $.
	
	Moreover, if $\alpha $ is a continuous action of $G$ on
	$\mathcal{M}$ by $\tau $-preserving automorphisms (i.e., $\alpha $ satisfies
	\emph{$(\mathbf{A} ')$}), then $(A_{r}x)_{r>0}$ converges a.u. to $Px$ as
	$r\to \infty $ for all $x\in L_{1}(\mathcal{M})$.
\end{thm}

\begin{proof}
	Note that for $1\leq p\leq 2$ and $2<p'<\infty $,
	$L_{p}(\mathcal{M})\cap L_{p'}(\mathcal{M})$ is dense in
	$L_{p}(\mathcal{M})$, and $(A_{r}x)_{r>0}$ converges a.u. to $Px$ for all
	$x\in L_{p'}(\mathcal{M})$ according to Proposition~\ref{prop:individual}. Then the theorem is an immediate consequence of
	Lemma~\ref{lem:principle litvinov} and Theorem~\ref{thm:max doubling}.
\end{proof}

For word metrics on groups of polynomial growth, it is well known that
the associated balls satisfy the F{\o}lner condition (see
\cite{breuillard14polygrowth,tessera07polygrowth}). Together with Corollary~\ref{cor:weakonepoly} we obtain the following result. This also proves
the a.u. convergence on $L_{1}(\mathcal{M})$ stated in Theorem~\ref{thm:main poly}.

\begin{thm}
	Assume that $G$ is of polynomial growth, and is generated by a symmetric
	compact subset~$V$. Let $\alpha $ be an action of $G$ well defined on
	$\bigcup _{1\leq p\leq \infty }L_{p}(\mathcal{M})$ which satisfies
	\emph{$(\mathbf{A} ^{p}_{1})$--$(\mathbf{A} ^{p}_{3})$} for every
	$1\leq p\leq \infty $. Denote
	\begin{equation*}
	A_{n}x=
	\frac{1}{m(V^{n})}
	\int _{V^{n}}\alpha _{s}x\,dm(s),\quad x\in
	L_{1}( \mathcal{M}),n\in \mathbb{N}.
	\end{equation*}
	Then $(A_{n} x)$ converges a.u. to $Px$ for all $1<p<\infty $.
	
	Moreover, if $\alpha $ is a continuous $\tau $-preserving action of
	$G$ on $\mathcal{M}$ such that $\alpha _{s}$ is a positive isometry on
	$\mathcal{M}$ for each $s\in G$, then $(A_{n} x)$ converges a.u. to
	$Px$ for all $x\in L_{1}(\mathcal{M})$.
\end{thm}

Also, it is obvious that an increasing sequence of compact subgroups always
satisfies the F{\o}lner condition. Together with Theorem~\ref{thm:max subgp} we obtain the following.
\begin{thm}
	Let $G$ be an increasing union of compact subgroups
	$(G_{n})_{n\geq 1}$. Let $\alpha $ be an action of $G$ well defined on
	$\bigcup _{1\leq p\leq \infty }L_{p}(\mathcal{M})$ which satisfies
	\emph{$(\mathbf{A} ^{p}_{1})$--$(\mathbf{A} ^{p}_{3})$} for every
	$1\leq p\leq \infty $. Denote
	\begin{equation*}
	A_{n}x=
	\frac{1}{m(G_{n})}
	\int _{G_{n}}\alpha _{s}x\,dm(s),\quad x\in
	L_{p}( \mathcal{M}),n\in \mathbb{N}.
	\end{equation*}
	Then $(A_{n} x)$ converges a.u. to $Px$ for all $1<p<\infty $.
	
	Moreover, if $\alpha $ is a continuous action of $G$ on
	$\mathcal{M}$ by $\tau $-preserving automorphisms (i.e., $\alpha $ satisfies
	\emph{$(\mathbf{A} ')$}), then $(A_{n} x)$ converges a.u. to $Px$ for all
	$x\in L_{1}(\mathcal{M})$.
\end{thm}

Note that all the above arguments rely on the assumption that the action
$\alpha $ extends to a uniformly bounded action on
$L_{\infty }(\mathcal{M})$ with conditions $(\mathbf{A} _{1}^{\infty })$--$(
\mathbf{A} _{3}^{\infty })$, though our strong type $(p,p)$ inequalities in
previous sections do not require this assumption. Also, in general this
assumption does not hold for bounded representations on one fixed
$L_{p}$-space. In Theorem~\ref{cor:individual balls} we will give a stronger
result for F{\o}lner sequences associated with doubling conditions. This
also completes the proof of Theorem~\ref{thm:main poly}.

\begin{lem}%
	\label{lem:dyadic}
	Let $(X,d,\mu )$ be a metric measure space satisfying the doubling condition
	\eqref{eq:measure doubling}. Take $i\in \mathbb{N}$ and $k\leq 2^{i}$. Then
	there exists $2^{i}\leq r_{i} <2^{i+1}$ such that
	\begin{equation*}
	\mu \bigl(B(x,r_{i}+k)\setminus B(x,r_{i}) \bigr)\leq Ck
	\mu \bigl(B(x,r_{i}) \bigr) / r_{i},
	\end{equation*}
	where $C$ depends only on the doubling constant.
\end{lem}
\begin{proof}
	The result and the argument are adapted from
	\cite[Proposition~17]{tessera07polygrowth}. For each
	$r\in \mathbb{N}$, we denote
	\begin{equation*}
	S(x,r)=B(x,r+k)\setminus B(x,r).
	\end{equation*}
	Then
	\begin{equation*}
	\bigcup _{n=0}^{[\frac{2^{i}}{k}]} S(x,2^{i}
	+nk) \subset B(x,2^{i+1}).
	\end{equation*}
	Therefore,
	\begin{equation*}
	\frac{2^{i}}{k}
	\inf _{0\leq n\leq [\frac{2^{i}}{k}]} \mu \bigl(S(x,2^{i}+nk) \bigr)
	\leq \mu \bigl(B(x,2^{i+1}) \bigr).
	\end{equation*}
	Thus, the lemma follows thanks to the doubling condition
	\eqref{eq:measure doubling}.
\end{proof}

\begin{lem}[{\cite[Theorem~1.1]{breuillard14polygrowth}, \cite[Corollary~10]{tessera07polygrowth}}]%
	\label{lem:tessera poly growth}
	Let $G$ be a locally
	compact group of polynomial growth, generated by a symmetric compact subset~$V$. Then
	\begin{equation*}
	\lim _{n\to \infty }
	\frac{m(V^{n})}{n^{d(G)}}=c,
	\end{equation*}
	where $d(G)$ is the rank of $G$ and $c$ is a constant depending on~$V$. And there exist $\delta >0$ and a constant $C$ such that
	\begin{equation*}
	\frac{m(V^{n+1}\setminus V^{n})}{m(V^{n})}\leq C n^{-\delta },\quad n \geq 1.
	\end{equation*}
\end{lem}

\begin{thm}
	\label{cor:individual balls}
	Fix $1< p <\infty $. Let $\alpha =(\alpha _{s})_{s\in G}$ be an action
	on $L_{p}(\mathcal{M})$ which satisfies
	\emph{$(\mathbf{A} ^{p}_{1})$--$(\mathbf{A} ^{p}_{3})$}.
	\begin{enumerate}
		\item[(1)] Assume that there exists an invariant metric $d$ on $G$ and
		that $(G,d)$ satisfies \eqref{eq:doubling def} and
		\eqref{eq:asymp inv}. Denote
		\begin{equation*}
		A_{r}x=
		\frac{1}{m(B_{r})}
		\int _{B_{r}}\alpha _{s}x\,dm(s),\quad x\in
		L_{p}( \mathcal{M}), r>0.
		\end{equation*}
		Then there exists a lacunary sequence $(r_{k})_{k\geq 1}$ with
		$2^{k}\leq r_{k} <2^{k+1}$ such that $(A_{r_{k}}x)_{k\geq 1}$ converges
		b.a.u. to $Px$ for all $x\in L_{p}(\mathcal{M})$. If additionally
		$p\geq 2$, then $(A_{r_{k}}x)_{k\geq 1}$ converges a.u. to $Px$ for all
		$x\in L_{p}(\mathcal{M})$.
		\item[(2)] Assume that $G$ is a locally compact group of polynomial growth,
		generated by a symmetric compact subset~$V$. Then the sequence
		\begin{equation*}
		A_{n}x=
		\frac{1}{m(V^{n})}
		\int _{V^{n}}\alpha _{s}x\,dm(s),\quad n\in \mathbb{N},
		\end{equation*}
		converges b.a.u. to $Px$ for all $x\in L_{p}(\mathcal{M})$.
		\item[(3)] Assume that $G$ is an increasing union of compact subgroups
		$(G_{n})_{n\geq 1}$. Then the sequence
		\begin{equation*}
		A_{n}x=
		\frac{1}{m(G_{n})}
		\int _{G_{n}}\alpha _{s}x\,dm(s),\quad n\in \mathbb{N},
		\end{equation*}
		converges a.u. to $Px$ for all $x\in L_{p}(\mathcal{M})$.
	\end{enumerate}
\end{thm}
\begin{proof}
	(1) By Lemma~\ref{lem:dyadic}, there exists $(r_{i})_{i\geq 1}$ such that
	$2^{i}\leq r_{i} \leq 2^{i+1}$ and such that
	\begin{equation*}
	m(B_{r_{i}}\setminus B_{r_{i}-(3/2)^{i}})\leq C(3/2)^{i}
	m(B_{r_{i}}) / r_{i}.
	\end{equation*}
	That is to say,
	\begin{equation}
	\label{eq:dydaic diff}
	\frac{m(B_{r_{i}}\setminus B_{r_{i}-(3/2)^{i}})}{m(B_{r_{i}})} \leq C (3/4)^{i}.
	\end{equation}
	
	We show that $(A_{r_{i}}x)_{i\geq 1}$ converges b.a.u. to $Px$. By
	\eqref{eq:decomposition erg} and Lemma~\ref{lem:czero}, it suffices to
	show that
	$(A_{r_{i}}x)_{i\geq 1}\in L_{p}(\mathcal{M};c_{0} )$ for
	$x\in \mathcal{F} _{p}^{\bot }$. By Theorem~\ref{thm:max doubling}, it is
	enough to consider the case where $x=y-\alpha _{s_{0}} y$ with
	$y\in L_{p}^{+}(\mathcal{M})$ and $s_{0}\in G$. Indeed, if
	$(A_{r_{i}}x)_{i\geq 1}\in L_{p}(\mathcal{M};c_{0} )$ for
	$x$'s of the aforementioned form $y-\alpha _{s_{0}} y$, then the same holds
	for all
	$x\in \operatorname{span} \{x-\alpha _{s}x:s\in G,x\in L_{p}(
	\mathcal{M})\}$. Now by the definition of
	$\mathcal{F} _{p}^{\bot }$, for any $\varepsilon >0$ and
	$x\in \mathcal{F} _{p}^{\bot }$, we may find an element
	$y\in \operatorname{span} \{x-\alpha _{s}x:s\in G,x\in L_{p}(
	\mathcal{M})\}$ with $ \llVert  x-y \rrVert  _{p}<\varepsilon $. So, by Theorem~\ref{thm:max doubling}, we see that
	\begin{equation*}
	\bigl\llVert (A_{r_{i}}x - A_{r_{i}}y) \bigr\rrVert
	_{L_{p}(\mathcal{M};\ell _{\infty })} \leq C \llVert x-y \rrVert _{p} <C\varepsilon
	\end{equation*}
	with a constant $C$ independent of $x$ and~$y$. Thus, $(A_{r_{i}}x)$ belongs
	to $L_{p}(\mathcal{M};c_{0} )$ since
	$L_{p}(\mathcal{M};c_{0} )$ is a Banach space. In the following
	we prove the claim
	$(A_{r_{i}}x)_{i\geq 1}\in L_{p}(\mathcal{M};c_{0} )$ with
	$x=y-\alpha _{s_{0}} y$. Denote $ \llvert  s_{0} \rrvert  =d(e,s_{0})$. Note that
	\begin{equation*}
	A_{r} x =A_{r}^{1} y
	-A_{r}^{2} y,
	\end{equation*}
	where
	\begin{align*}
	A_{r}^{1} y&=
	\frac{1}{m(B_{r})}
	\int _{B_{r}\setminus (B_{r}\cap B_{r} s_{0})} \alpha _{s} y \,dm(s),\quad\quad
	\\A_{r}^{2} y&=
	\frac{1}{m(B_{r})}
	\int _{(B_{r} s_{0})
		\setminus (B_{r}\cap B_{r} s_{0})}\alpha _{s} y \,dm(s).
	\end{align*}
	By \eqref{eq:dydaic diff} we have for $i$ so that
	$(3/2)^{i}\geq  \llvert  s_{0} \rrvert  $,
	\begin{equation*}
	\llVert A_{r_{i}}^{1} y \rrVert _{p}\leq
	\frac{m(B_{r_{i}}\setminus B_{r_{i}- \llvert  s_{0} \rrvert  })}{m(B_{r_{i}})} \llVert y \rrVert _{p} \leq C (3/4)^{i}
	\llVert y \rrVert _{p}.
	\end{equation*}
	On the other hand, for any $m\leq j\leq n$,
	\begin{equation*}
	A_{r_{j}}^{1} y = \bigl[(A_{r_{j}}^{1}
	y )^{p} \bigr]^{1/p} \leq \Bigl[
	\sum _{m
		\leq i\leq n} (A_{r_{i}}^{1} y
	)^{p} \Bigr]^{1/p},
	\end{equation*}
	and by the previous argument
	\begin{equation*}
	\Bigl\llVert \Bigl[
	\sum _{m\leq i\leq n} A_{r_{i}}^{1} y
	^{p} \Bigr]^{1/p} \Bigr\rrVert _{p}
	\leq \Bigl[
	\sum _{m\leq i\leq n} \llVert A_{r_{i}}^{1} y
	\rrVert _{p} ^{p} \Bigr]^{1/p}\leq
	C(3/4)^{m} \llVert y \rrVert _{p}.
	\end{equation*}
	Hence,
	$ \llVert  (A_{r_{j}}^{1} y)_{m\leq j\leq n} \rrVert  _{L_{p}(\mathcal{M};c_{0}
		)}$ tends to $0$ as $m,n\to \infty $. Similarly,
	$(A_{r_{i}}^{2} y)_{i\geq 1}$ converges in the same manner. Therefore,
	$(A_{r_{i}}x)_{i\geq 1}\in L_{p}(\mathcal{M};c_{0} )$, as desired.
	
	Moreover, if $p \geq 2$, then
	\begin{equation*}
	(A_{r_{j}}^{1} y )^{2} =
	\bigl[(A_{r_{j}}^{1} y )^{p}
	\bigr]^{2/p} \leq \Bigl[
	\sum _{m\leq i\leq n} (A_{r_{i}}^{1} y
	)^{p} \Bigr]^{2/p},
	\end{equation*}
	and hence we can find contractions
	$u_{j}\in L_{\infty }(\mathcal{M}) $ such that for $m$ large enough,
	\begin{equation*}
	A_{r_{j}}^{1} y =u_{j} \Bigl[
	\sum _{m\leq i\leq n} (A_{r_{i}}^{1} y
	)^{p} \Bigr]^{1/p}\quad \text{with } \Bigl
	\llVert \Bigl[
	\sum _{m\leq i\leq n} (A_{r_{i}}^{1} y
	)^{p} \Bigr]^{1/p} \Bigr\rrVert
	_{p}\leq C (3/4)^{m} \llVert y \rrVert
	_{p}.
	\end{equation*}
	Therefore,
	$ \llVert  (A_{r_{j}}^{1} y)_{m\leq j\leq n} \rrVert  _{L_{p}(\mathcal{M};c_{0}^{c})}$
	tends to $0$ as $m,n\to \infty $, and $(A_{r_{i}}^{1} y)_{i\geq 1}$ converges
	a.u. to $0$ according to Lemma~\ref{lem:czero}. Similarly,
	$(A_{r_{i}}^{2} y)_{i\geq 1}$ converges in the same manner. Thus, we obtain
	that $(A_{r_{i}}x)_{i\geq 1}$ converges a.u. to~$0$. Then by
	\eqref{eq:decomposition erg} and Lemma~\ref{lem:principle litvinov},
	$(A_{r_{i}}x)_{i\geq 1}$ converges a.u. for all
	$x\in L_{p}(\mathcal{M})$.
	
	(2) We keep the notation $x$, $y$, $s_{0}$ in (1), and denote as before
	\begin{equation*}
	A_{k}^{1} y=
	\frac{1}{m(V^{k})}
	\int _{(V^{k})\setminus (V^{k}\cap V^{k}
		s_{0})}\alpha _{s} y \,dm(s) ,\quad y\in
	L_{p}(\mathcal{M}),k \geq 1.
	\end{equation*}
	Write $\delta '=[\delta ^{-1}]+1$. Note that by Lemma~\ref{lem:tessera poly growth}, there exists a constant $C>0$ such that
	for $y\in L_{p}(\mathcal{M})$ and $k\geq 1$,
	\begin{equation}
	\label{eq:ndelta} \llVert A_{k^{\delta '}}^{1} y \rrVert
	_{p} \leq
	\frac{m(V^{k^{\delta '}}\setminus V^{k^{\delta '}- \llvert  s_{0} \rrvert  })}{m(V^{k^{\delta '}})} \llVert y \rrVert _{p}\leq
	\frac{C \llvert  s_{0} \rrvert   \llVert  y \rrVert  _{p}}{k},
	\end{equation}
	where $ \llvert  s_{0} \rrvert  =d(e,s_{0})$ and $d$ refers to the word metric defined in
	Example~\ref{ex:max doubling}(1). Hence,
	\begin{equation*}
	\Bigl\llVert \Bigl[
	\sum _{m\leq k\leq n} A_{k^{\delta '}}^{1} y
	^{p} \Bigr]^{1/p} \Bigr\rrVert _{p}
	\leq \Bigl[
	\sum _{m\leq k\leq n} \llVert A_{k^{\delta '}}^{1} y
	\rrVert _{p} ^{p} \Bigr]^{1/p}\leq C
	\llvert s_{0} \rrvert \Bigl(
	\sum _{m\leq k\leq n}
	\frac{1}{k^{p}} \Bigr)^{1/p} \llVert y \rrVert _{p}.
	\end{equation*}
	Then, by an argument similar to that in (1), we see that
	$(A_{n^{\delta '}} x)_{n\geq 1}$ converges b.a.u. to $Px$ for all
	$x\in L_{p}(\mathcal{M})$, and if $p\geq 2$, then
	$(A_{n^{\delta '}} x)_{n\geq 1}$ converges a.u. to $Px$.
	
	For the general case, we consider
	$x\in L_{p}^{+}(\mathcal{M})$. For each $k$, let $n(k)$ be the
	number such that $n(k)^{\delta '}\leq k< (n(k)+1)^{\delta '}$. Then
	\begin{equation*}
	\frac{m(V^{n(k)^{\delta '}})}{m(V^{k})} A_{n(k)^{\delta '}} x \leq A_{k} x\leq
	\frac{m(V^{(n(k)+1)^{\delta '}})}{m(V^{k})} A_{(n(k)+1)^{
			\delta '}} x.
	\end{equation*}
	Also note that according to Lemma~\ref{lem:tessera poly growth},
	$m(V^{n(k)^{\delta '}})/m(V^{k})$ tends to~$1$. Therefore, it is easy to
	see from the definition of b.a.u. convergence that $A_{k} x$ converges
	b.a.u. to $Px$.\vadjust{\goodbreak}
	
	(3) Note that for $x=y-\alpha _{s_{0}} y$ with
	$y\in L_{p}(\mathcal{M})$ and $s_{0}\in G$, we have for
	$n$ large enough so that $G_{n}\ni s_{0}$,
	\begin{equation*}
	A_{n}x=
	\frac{1}{m(G_{n})}
	\int _{G_{n}}\alpha _{s} y \,dm(s)-
	\int _{G_{n}} \alpha _{s s_{0}}y\,dm(s)=0.
	\end{equation*}
	That is to say, $A_{n} x$ converges a.u. to $0$ as $n\to \infty $. Then
	by \eqref{eq:decomposition erg} and Lemma~\ref{lem:principle litvinov}, we see that $A_{n} x$ converges a.u. for
	all $x\in L_{p}(\mathcal{M})$.
\end{proof}

In particular, the above arguments give the individual ergodic theorem
for positive invertible operators on $L_{p}$-spaces.
\begin{cor}
	Let $1<p<\infty $. Let $T:L_{p}(\mathcal{M}) \to L_{p}(\mathcal{M})$ be
	a positive invertible operator with positive inverse such that
	$\sup _{n\in \mathbb{Z}} \llVert  T^{n} \rrVert  <\infty $. Denote
	\begin{equation*}
	A_{n}=
	\frac{1}{2n+1}
	\sum _{k=-n}^{n}T^{k},\quad
	n\in \mathbb{N}.
	\end{equation*}
	Then $(A_{n}x)_{n\geq 1}$ converges b.a.u. to $Px$ for all
	$x\in L_{p}(\mathcal{M})$. If additionally $p \geq 2$, then
	$(A_{n}x)_{n\geq 1}$ converges a.u. to $Px$.
\end{cor}
\begin{proof}
	The assertion follows from the proof of Theorem~\ref{cor:individual balls}(2) for the case where $G$ equals the integer
	group $\mathbb{Z}$. It suffices to notice that in this case we may choose
	$\delta =1$ in Lemma~\ref{lem:tessera poly growth} and take
	$\delta '=1$ in \eqref{eq:ndelta}.
\end{proof}
\begin{rem}
	Note that the above result is not true for $p=1$, even for positive invertible
	isometries on classical $L_{1}$-spaces (see, e.g.,
	\cite{ionescutulcea64ergisometry}). So it is natural to assume that
	$p\neq 1$ in the above discussions.
\end{rem}

The following conjecture for mean bounded maps is still open. The result
for classical $L_{p}$-spaces is given by
\cite{martinreyesdelatorre88ergpowerbdd}.
\begin{conj}
	Let $1<p<\infty $. Let $T:L_{p}(\mathcal{M}) \to L_{p}(\mathcal{M})$ be
	a positive invertible operator with positive inverse such that
	$\sup _{n\in \mathbb{Z}} \llVert  \frac{1}{2n+1}\sum _{k=-n}^{n} T^{k} \rrVert  <
	\infty $. Denote
	\begin{equation*}
	A_{n}=
	\frac{1}{2n+1}
	\sum _{k=-n}^{n}T^{k},\quad
	n\in \mathbb{N}.
	\end{equation*}
	Then $(A_{n}x)_{n\geq 1}$ converges b.a.u. to $Px$ for all
	$x\in L_{p}(\mathcal{M})$. If additionally $p \geq 2$, then
	$(A_{n}x)_{n\geq 1}$ converges a.u. to $Px$.
\end{conj}


\subsection*{Acknowledgment}
	The subject of this paper came from a suggestion made several years ago
	to the authors by Professor Quanhua Xu, whom we thank for many fruitful
	discussions. We would also like to thank the anonymous referees for their
	helpful comments and suggestions. Part of the work was done during the
	Summer Working Seminar on Noncommutative Analysis at Wuhan University (2015)
	and the Institute for Advanced Study in Mathematics at Harbin Institute
	of Technology (2016 and 2017).
	
	This work was supported by National Science
	Foundation (NSF) of China grant 11431011. Hong's work was partially supported
	by NSF of China grants 11601396 and 12071355. Liao's work was partially
	supported by the National Science Foundation and by Coalition for National
	Science Funding grant 11420101001. Wang's work was partially supported
	by the European Research Council Advanced Grant ``Non-Commutative Distributions
	in Free Probability'' No. 339760, Agence Nationale de la Recherche grant
	ANR-19-CE40-0002, and a public grant as part of the Fondation Math\'{e}matique
	Jacques Hadamard.

\newcommand{\etalchar}[1]{$^{#1}$}

\end{document}